\documentclass[12pt]{amsart}

\usepackage{amssymb,amsthm,amsmath,amsxtra,color}
\usepackage{hyperref}
\usepackage[all]{xy}
\usepackage{fullpage}
\usepackage{comment}
\usepackage{multirow}

\numberwithin{equation}{section}

\theoremstyle{plain}
\newtheorem{thm}[equation]{Theorem}
\newtheorem{prop}[equation]{Proposition}
\newtheorem{lem}[equation]{Lemma}
\newtheorem{cor}[equation]{Corollary}

\newtheorem*{cor*}{Corollary}
\newtheorem*{prob*}{Problem}
\newtheorem*{thm*}{Theorem}
\newtheorem*{thma*}{Theorem A}
\newtheorem*{thmb*}{Theorem B}
\newtheorem*{thmc*}{Theorem C}

\theoremstyle{remark}
\newtheorem{exm}[equation]{Example}

\newtheorem{rmk}[equation]{Remark}

\newenvironment{enumalph}
{\begin{enumerate}}
{\end{enumerate}}

\newenvironment{enumroman}
{\begin{enumerate}}
{\end{enumerate}}

\DeclareMathOperator{\Aut}{Aut}

\DeclareMathOperator{\Frob}{Frob}
\DeclareMathOperator{\Gal}{Gal}
\DeclareMathOperator{\Inn}{Inn}
\DeclareMathOperator{\M}{M}

\DeclareMathOperator{\opdiv}{div}
\DeclareMathOperator{\nrd}{nrd}

\DeclareMathOperator{\Out}{Out}
\DeclareMathOperator{\PGL}{PGL}
\DeclareMathOperator{\PSL}{PSL}

\DeclareMathOperator{\SL}{SL}
\DeclareMathOperator{\GL}{GL}
\DeclareMathOperator{\tr}{tr}

\DeclareMathOperator{\Ker}{Ker}

\newcommand{\C}{\mathbb C}
\newcommand{\F}{\mathbb F}

\newcommand{\PP}{\mathbb P}
\newcommand{\Q}{\mathbb Q}
\newcommand{\R}{\mathbb R}
\newcommand{\Z}{\mathbb Z}
\newcommand{\Qbar}{\overline{\mathbb Q}}

\newcommand{\Belyi}{Bely\u{\i}}

\newcommand{\frakd}{\mathfrak{d}}
\newcommand{\frakn}{\mathfrak{n}}
\newcommand{\frakM}{\mathfrak{M}}
\newcommand{\frakN}{\mathfrak{N}}
\newcommand{\frakp}{\mathfrak{p}}
\newcommand{\frakP}{\mathfrak{P}}

\newcommand{\calH}{\mathcal{H}}
\newcommand{\calM}{\mathcal{M}}
\newcommand{\calO}{\mathcal{O}}

\newcommand{\psmod}[1]{~(\textup{\text{mod}}~{#1})}

\newcommand{\quat}[2]{\displaystyle{\biggl(\frac{#1}{#2}\biggr)}}

\newcommand{\la}{\langle}
\newcommand{\ra}{\rangle}

\newcommand{\gunder}{\underline{g}}
\newcommand{\tunder}{\underline{t}}
\newcommand{\Cunder}{\underline{C}}

\DeclareMathOperator{\opchar}{char}

\newcommand{\Deltabar}{\overline{\Delta}}
\newcommand{\deltabar}{\bar{\delta}}

\newcommand{\Gammabar}{\overline{\Gamma}}

\newcommand{\sbwk}{{}_{\textup{wr}}}

\newcommand{\sbrat}{{}_{\textup{r}}}

\newcommand{\sbtors}{{}_{\textup{tors}}}

\newcommand{\Bhat}{\widehat{B}}

\newcommand{\calOhat}{\widehat{\mathcal O}}

\newcommand{\defi}[1]{\textsf{#1}}

\begin{document}

\title[Congruence subgroups of triangle groups]{Algebraic curves uniformized by congruence subgroups of triangle groups}

\author{Pete L.\ Clark}
\address{Department of Mathematics, University of Georgia, Athens, GA 30602, USA}
\email{pete@math.uga.edu}

\author{John Voight}
\address{Department of Mathematics and Statistics, University of Vermont, 16 Colchester Ave, Burlington, VT 05401, USA; Department of Mathematics,
  Dartmouth College, 6188 Kemeny Hall, Hanover, NH 03755, USA}
\email{jvoight@gmail.com}
\date{\today}

\begin{abstract}
We construct certain subgroups of hyperbolic triangle groups which we call ``congruence'' subgroups.  These groups include the classical congruence subgroups of $\SL_2(\Z)$, Hecke triangle groups, and $19$ families of arithmetic triangle groups associated to Shimura curves.  We determine the field of moduli of the curves associated to these groups and thereby realize the groups $\PSL_2(\F_q)$ and $\PGL_2(\F_q)$ regularly as Galois groups.
\end{abstract}

\maketitle
\tableofcontents

\section{Introduction}

\subsection*{Motivation}

The rich arithmetic and geometric theory of classical modular curves, quotients of the upper half-plane by subgroups of $\SL_2(\Z)$ defined by congruence conditions, has fascinated mathematicians since at least the nineteenth century.  One can see these curves as special cases of several distinguished classes of curves.  Fricke and Klein \cite{FK} investigated curves obtained as quotients by Fuchsian groups which arise from the unit group of certain quaternion algebras, now called arithmetic groups.  Later, Hecke \cite{Hecke} investigated his triangle groups, arising from reflections in the sides of a hyperbolic triangle with angles $0,\pi/2,\pi/n$ for $n \geq 3$.  Then in the 1960s, amidst a flurry of activity studying the modular curves, Atkin and Swinnerton-Dyer \cite{ASD} pioneered the study of noncongruence subgroups of $\SL_2(\Z)$.  In this paper, we consider a further direction: we introduce a class of curves arising from certain subgroups of hyperbolic triangle groups.  These curves share many appealing properties in common with classical modular curves despite the fact that their uniformizing Fuchsian groups are in general \emph{not} arithmetic groups.

To motivate the definition of this class of curves, we begin with the modular curves.  Let $p$ be prime and let $\Gamma(p) \subseteq \PSL_2(\Z)=\Gamma(1)$ be the subgroup of matrices congruent to (plus or minus) the identity modulo $p$.  Then $\Gamma(p)$ acts on the completed upper half-plane $\calH^*$, and the quotient $X(p)=\Gamma(p) \backslash \calH^*$ can be given the structure of Riemann surface, a modular curve.  The subgroup $G=\Gamma(1)/\Gamma(p) \subseteq \Aut(X(p))$ satisfies $G \cong \PSL_2(\F_p)$ and the natural map $j:X(p) \to X(p)/G \cong \PP_\C^1$ is a Galois branched cover ramified at the points $\{0,1728,\infty\}$.

So we are led to study class of (smooth, projective) curves $X$ over $\C$ with the property that there exists a subgroup $G \subseteq \Aut(X)$ with $G \cong \PSL_2(\F_q)$ or $G \cong \PGL_2(\F_q)$ for a prime power $q$ such that the map $X \to X/G \cong \PP^1$ is a Galois branched cover ramified at exactly three points.

\Belyi\ \cite{Belyi,Belyi2} proved that a curve $X$ over $\C$ can be defined over the algebraic closure $\Qbar$ of $\Q$ if and only if $X$ admits a \defi{\Belyi\ map}, a nonconstant morphism $f:X \to \PP^1_\C$ unramified away from $0,1,\infty$.  So, on the one hand, three-point branched covers are indeed ubiquitous.  On the other hand, there are only \emph{finitely} many curves $X$ up to isomorphism of given genus $g \geq 2$ which admit a \emph{Galois} \Belyi\ map (Remark \ref{onlyfinGaloisBelyi}).  We call a curve which admits a Galois \Belyi\ map a \defi{Galois \Belyi\ curve}.  Galois \Belyi\ curves are also called \defi{quasiplatonic surfaces}  \cite{GirondoWolfart,Wolfart97}, owing to their connection with the Platonic solids, or \defi{curves with many automorphisms} because they are equivalently characterized as the locus on the moduli space $\calM_g(\C)$ of curves of genus $g$ at which the function $[C] \mapsto \#\Aut(C)$ attains a strict local maximum.  For example, the Hurwitz curves, those curves $X$ with maximal automorphism group $\#\Aut(X)=84(g-1)$ for their genus $g$, are Galois \Belyi\ curves, as are the Fermat curves $x^n+y^n=z^n$ for $n \geq 3$.

So Galois \Belyi\ curves with Galois group $G=\PSL_2(\F_q)$ and $G=\PGL_2(\F_q)$ generalize the classical modular curves and bear further investigation.  In this article, we study the existence of these curves, and we then consider one of the most basic properties about them: the fields over which they are defined.

\subsection*{Existence}

To state our first result concerning existence we use the following notation.  For $s \in \Z_{\geq 2}$, let $\zeta_s=\exp(2\pi i/s)$ and $\lambda_s = \zeta_{s}+1/\zeta_{s}=2\cos(2\pi/s)$; by convention we let $\zeta_\infty=1$ and $\lambda_\infty=2$.

Let $a,b,c \in \Z_{\geq 2} \cup \{\infty\}$ satisfy $a \leq b \leq c$.  Then we have the following extension of fields:
\begin{equation} \tag{*}
\begin{aligned}
\xymatrix{
F(a,b,c)=\Q(\lambda_{2a},\lambda_{2b},\lambda_{2c}) \ar@{-}[d] \\
E(a,b,c)=\Q(\lambda_a,\lambda_b,\lambda_c,\lambda_{2a}\lambda_{2b}\lambda_{2c}) \ar@{-}[d] \\
D(a,b,c)=\Q(\lambda_a,\lambda_b,\lambda_c) \ar@{-}[d] \\
\Q
}
\end{aligned}
\end{equation}
We have $E(a,b,c) \subseteq F(a,b,c)$ since $\lambda_{2a}^2=\lambda_a+2$ (the half-angle formula), and consequently this extension has degree at most $4$ and exponent at most $2$.  Accordingly, a prime $\frakp$ of $E(a,b,c)$ (by which we mean a nonzero prime ideal in the ring of integers of $E(a,b,c)$) that is unramified in $F(a,b,c)$ either splits completely or has inertial degree $2$.

The triple $(a,b,c)$ is \defi{hyperbolic} if
\[ \chi(a,b,c)=\frac{1}{a}+\frac{1}{b}+\frac{1}{c}-1 < 0. \]

Our first main result is as follows.

\begin{thma*}
Let $(a,b,c)$ be a hyperbolic triple with $a,b,c \in \Z_{\geq 2}$.  Let $\frakp$ be a prime of $E(a,b,c)$ with residue field $\F_{\frakp}$ and suppose $\frakp \nmid 2abc$.  Then there exists a $G$-Galois \Belyi\ map
\[ X(a,b,c;\frakp) \to \PP^1 \]
with ramification indices $(a,b,c)$, where
\[ G=
\begin{cases}
\PSL_2(\F_\frakp), & \text{if $\frakp$ splits completely in $F(a,b,c)$}; \\
\PGL_2(\F_\frakp), &\text{otherwise}.
\end{cases}
\]

\end{thma*}

We have stated Theorem A in a simpler form; for a more general statement, including the case when $\frakp \mid 2$ or when one or more of $a,b,c$ is equal to $\infty$, see Theorem \ref{thmaorders}.  In some circumstances (depending on a norm residue symbol), one can also take primes dividing $abc$ (see Remark \ref{TheoremBext}).

Theorem A generalizes work of Lang, Lim, and Tan \cite{LangLimTan} who treat the case of Hecke triangle groups using an explicit presentation of the group (see also Example \ref{exmLLT}), and work of Marion \cite{Marion} who treats the case $a,b,c$ prime.  This also complements celebrated work of Macbeath \cite{Macbeath}, providing an explicit way to distinguish between projective two-generated subgroups of $\PSL_2(\F_q)$ by a simple splitting criterion.  Theorem A overlaps with work of Conder, Poto\v{c}ink, and \v{S}ir\'a\v{n} \cite{CPS} (they also give several other references to work of this kind).

The construction of Galois \Belyi\ maps in Theorem A arises from another equivalent characterization of Galois \Belyi\ curves (of genus $\geq 2$) as compact Riemann surfaces of the form $\Gamma \backslash \calH$, where $\Gamma$ is a finite index normal subgroup of the hyperbolic triangle group
\[ \Deltabar(a,b,c)=\langle \deltabar_a,\deltabar_b,\deltabar_c \mid \deltabar_a^a=\deltabar_b^b=\deltabar_c^c=\deltabar_a\deltabar_b\deltabar_c=1 \rangle \subset \PSL_2(\R) \]
for some $a,b,c \in \Z_{\geq 2}$, where by convention we let $\deltabar_\infty^\infty=1$.  (See Sections 1--2 for more detail.  The bars may seem heavy-handed here, but signs play a delicate and somewhat important role in the development, so we include this as part of the notation for emphasis.)  Phrased in this way, Theorem A asserts the existence of a normal subgroup
\[ \Deltabar(\frakp)=\Deltabar(a,b,c;\frakp) \trianglelefteq \Deltabar(a,b,c)=\Deltabar \] with quotient $\Deltabar/\Deltabar(\frakp)=G$ as above.  In a similar way, one obtains curves $X_0(a,b,c;\frakp)$ by considering the quotient of $X(a,b,c;\frakp)$ by the subgroup of upper-triangular matrices---these curves are analogous to the classical modular curves $X_0(p)$ as quotients of $X(p)$.

\subsection*{Arithmeticity}

A Fuchsian group is \defi{arithmetic} if it is commensurable with the group of units of reduced norm $1$ of a maximal order in a quaternion algebra.  A deep theorem of Margulis \cite{Margulis} states that a Fuchsian group is arithmetic if and only if it is of infinite index in its commensurator.  By work of Takeuchi \cite{Takeuchi}, only finitely many of the groups $\Delta(a,b,c)$ are arithmetic: in these cases, the curves $X(a,b,c;\frakp)$ are \defi{Shimura curves} (arising from full congruence subgroups) and canonical models were studied by Shimura \cite{Shimura} and Deligne \cite{Deligne}.  Indeed, the curves $X(2,3,\infty;p)$ are the classical modular curves $X(p)$ and the Galois \Belyi\ map $j:X(p) \to \PP^1$ is associated to the congruence subgroup $\Gamma(p) \subseteq \PSL_2(\Z)$.  Several other arithmetic families of Galois \Belyi\ curves have seen more detailed study, most notably the family $X(2,3,7;\frakp)$ of Hurwitz curves.  (It is interesting to note that the arithmetic triangle groups are among the examples given by Shimura \cite[Example 3.18]{Shimura}!)  Aside from these finitely many triples, the triangle group $\Delta=\Delta(a,b,c)$ is \emph{not} arithmetic, and our results can be seen as a generalization in this nonarithmetic context.

However, we still have an embedding inside an arithmetic group $\Gamma$, following Takeuchi \cite{Takeuchi} and later work of Tretkoff (n\'ee Cohen) and Wolfart \cite{CohenWolfart}: our curves are obtained via pullback
\[
\xymatrix{
\Deltabar(\frakp) \backslash \calH \ar[r] \lhook\mkern-7mu \ar[d] & \Gamma(\frakp) \backslash \calH^s \ar[d] \\
\PP^1=\Deltabar \backslash \calH \lhook\mkern-7mu \ar[r] & \Gamma(1) \backslash \calH^s
} \]
from a branched cover of quaternionic Shimura varieties, and this promises further arithmetic applications.  Accordingly, we call the subgroups $\Deltabar(a,b,c;\frakp)$ we construct \defi{congruence} subgroups of $\Deltabar$ in analogy with the classical case of modular curves, since they arise from certain congruence conditions on matrix entries.  (In some contexts, the term \emph{congruence} is used only for arithmetic groups; we propose the above extension of this terminology to non-arithmetic groups.)  For a fuller discussion of the arithmetic cases of Theorem A, see Example \ref{Shimurafielddef}.

\subsection*{Field of definition}

Our second main result studies fields of definition.  The modular curve $X(p)$, a Riemann surface defined over $\C$, has a model as an algebraic curve defined over $\Q$; we seek a similar (nice, explicit) result for our class of curves.  For a curve $X$ defined over $\C$, the \defi{field of moduli} $M(X)$ of $X$ is the fixed field of the group $\{\sigma \in \Aut(\C) : X^{\sigma} \cong X\}$, where $X^{\sigma}$ is the base change of $X$ by the automorphism $\sigma \in \Aut(\C)$.  A field of definition for $X$ clearly contains the field of moduli of $X$, so if $X$ has a minimal field of definition $F$ then $F$ is necessarily equal to the field of moduli.

We will need two refinements of this notion.  First, we define the notion for branched covers.  We say that two \Belyi\ maps $f:X \to \PP^1$ and $f':X' \to \PP^1$ are \defi{isomorphic} (over $\C$ or $\Qbar$) if there exists an isomorphism $h:X \xrightarrow{\sim} X'$ that respects the branched covers, i.e., such that $f=f'\circ h$.  We define the field of moduli $M(X,f)$ of a \Belyi\ map analogously.  A Galois \Belyi\ map can always be defined over its field of moduli (Lemma \ref{canbedefined}) as \defi{mere cover}.

But we will also want to keep track of the Galois automorphisms of the branched cover.  For a finite group $G$, a \defi{$G$-Galois \Belyi\ map} is a \Belyi\ map $f:X \to \PP^1$ equipped with an isomorphism $i:G \xrightarrow{\sim} \Gal(f)$ between $G$ and the Galois (monodromy) group of $f$, and an isomorphism of $G$-Galois \Belyi\ maps is an isomorphism $h$ of \Belyi\ maps that identifies $i$ with $i'$, i.e.,
\begin{center}
$h(i(g)x)=i'(g)h(x)$ for all $g \in G$ and $x \in X(\C)$
\end{center}
so the diagram
\[
\xymatrix{
X \ar[d]_{i(g)} \ar[rr]^{h} & & X' \ar[d]^{i'(g)} \\
X \ar[dr]_{f} \ar[rr]^{h} & & X' \ar[dl]^{f'} \\
& \PP^1
} \]
commutes.  We define the field of moduli $M(X,f,G)$ of a $G$-Galois \Belyi\ map $f$ accordingly.    For example, we have
\[ M(X(p),j,\PSL_2(\F_p))=\Q(\sqrt{p^*}) \quad \text{where $p^*=(-1)^{(p-1)/2} p$}. \]
A $G$-Galois \Belyi\ map $f$ can be defined over its field of moduli $M(X,f,G)$ under the following condition: if $G$ has trivial center $Z(G)=\{1\}$ and $G=\Aut(X)$ (otherwise, take a further quotient).

On the one hand, we observe (Remark \ref{allWolfart}) that for any number field $K$ there is a $G$-Galois \Belyi\ map $f$ for some finite group $G$ such that the field of moduli of $(X,f,G)$ contains $K$.  On the other hand, we will show that our curves have quite nice fields of definition.  (See also work of Streit and Wolfart \cite{StreitWolfart} who consider the case $G \cong \Z/p\Z \rtimes \Z/q\Z$.)

We need one further bit of notation.  For a prime $p$ and integers $a,b,c \in \Z_{\geq 2}$, let $D_{p'}(a,b,c)$ be the compositum of the fields $\Q(\lambda_s)$ with $s \in \{a,b,c\}$ prime to $p$.  (For example, if all of $a,b,c$ are divisible by $p$, then $D_{p'}(a,b,c)=\Q$.)  Similarly define $F_{p'}(a,b,c)$.

\begin{thmb*}
Let $X$ be a curve of genus $g \geq 2$ and let $f:X \to \PP^1$ be a $G$-Galois \Belyi\ map with $G\cong \PGL_2(\F_q)$ or $G \cong \PSL_2(\F_q)$.  Let $(a,b,c)$ be the ramification indices of $f$.

Then the following statements hold.

\begin{enumalph}
\item Let $r$ be the order of $\Frob_p$ in $\Gal(F_{p'}(a,b,c)/\Q)$.  Then
\[ q = \begin{cases}
\sqrt{p^r}, & \text{ if $G \cong \PGL_2(\F_q)$;} \\
p^r, & \text{ if $G \cong \PSL_2(\F_q)$.}
\end{cases} \]
\item The map $f$ as a mere cover is defined over its field of moduli $M(X,f)$.  Moreover, $M(X,f)$ is an extension of $D_{p'}(a,b,c)^{\la \Frob_p \ra}$ of degree $d_{(X,f)} \leq 2$.  If $a=2$ or $q$ is even, then $d_{(X,f)}=1$.
\item The map $f$ as a $G$-Galois \Belyi\ map is defined over its field of moduli $M(X,f,G)$.  Let
\[ D_{p'}(a,b,c)\{\sqrt{p^*}\}=\begin{cases}
D_{p'}(a,b,c)(\sqrt{p^*}), &\text{ if $p \mid abc$, $pr$ is odd, and $G \cong \PSL_2(\F_q)$}; \\
D_{p'}(a,b,c) &\text{ otherwise.}
\end{cases} \]
Then $M(X,f,G)$ is an extension of $D_{p'}(a,b,c)\{\sqrt{p^*}\}$ of degree $d_{(X,f,G)} \leq 2$.  If $q$ is even or $p \mid abc$ or $G \cong \PGL_2(\F_q)$, then $d_{(X,f,G)} =1$.
\end{enumalph}
\end{thmb*}

(Not all $G$-Galois \Belyi\ maps with $G \cong \PGL_2(\F_q)$ or $G \cong \PSL_2(\F_q)$ arise from the construction in Theorem A, but Theorem B applies to them all.)

The various fields of moduli fit into the following diagram.
\[
\xymatrix@C=0pt{
& M(X,G,f) \ar@{-}[d]^{d_{(X,f,G)}\leq 2} \ar@{-}[dl]  \\
M(X,f) \ar@{-}[d]_{d_{(X,f)}\leq 2} & D_{p'}(a,b,c)\{\sqrt{p^*}\}=\Q(\lambda_a,\lambda_b,\lambda_c)_{p'}\{\sqrt{p^*}\} \ar@{-}[dl] \\
D_{p'}(a,b,c)^{\la \Frob_p \ra}=\Q(\lambda_a,\lambda_b,\lambda_c)_{p'}^{\la \Frob_p \ra}
} \]

As a simple special case of Theorem B, we have the following corollary.

\begin{cor*}
Suppose $f:X \to \PP^1$ is a $\PSL_2(\F_q)$-Galois \Belyi\ map with ramification indices $(2,3,c)$ and suppose $p \nmid 6c$ is prime and a primitive root modulo $2c$.  Then $q=p^r$ where $r=\phi(2c)/2$ and $f$ is defined over $\Q$.  Moreover, the monodromy group $\Gal(f)$ is defined over an at most quadratic extension of $\Q(\lambda_{p})$.
\end{cor*}

To prove Theorem B, we use a variant of the rigidity and rationality results which arise in the study of the inverse Galois problem  \cite{MalleMatzat,Volklein} and apply them to the groups $\PSL_2(\F_q)$ and $\PGL_2(\F_q)$.  We use the classification of subgroups of $\PSL_2(\F_q)$ generated by two elements provided by Macbeath \cite{Macbeath}.  The statements $q=\sqrt{p^r}$ and $q=p^r$, respectively, can be found in earlier work of Langer and Rosenberger \cite[Satz (4.2)]{LangerRosenberger}; our proof follows similar lines.  Theorem B generalizes work of Schmidt and Smith \cite[Section 3]{SchmidtSmith} who consider the case of Hecke triangle groups as well as work of Streit \cite{Streit} and D{\v{z}}ambi{\'c} \cite{Dzambic} who considers Hurwitz groups, where $(a,b,c)=(2,3,7)$.

\subsection*{Composite level}

The congruence subgroups so defined naturally extend to composite ideals, and so they form a projective system (Proposition \ref{densesub}).  For a prime $\frakp$ of $E$ and $e \geq 1$, let $P(\frakp^e)$ be the group
\[ P(\frakp^e)=
\begin{cases}
\PSL_2(\Z_E/\frakp^e), &\text{ if $\frakp$ splits completely in $F$;} \\
\PGL_2(\Z_E/\frakp^e), &\text{ otherwise}
\end{cases} \]
where $\Z_E$ denotes the ring of integers of $E$.  For an ideal $\frakn$ of $\Z_E$, let $P(\frakn)=\prod_{\frakp^e \parallel \frakn} P(\frakp^e)$, and let $\widehat{P}=\varprojlim_{\frakn} P(\frakn)$ be the projective limit of $P(\frakn)$ with respect to the ideals $\frakn$ with $\frakn \nmid 6abc$.

\begin{thmc*}
$\Deltabar(a,b,c)$ is dense in $\widehat{P}$.
\end{thmc*}

Kucharczyk \cite{Kucharcyzk} uses superstrong approximation for thin subgroups of arithmetic groups to prove a version of Theorem C that shows that the closure of the image of $\Deltabar(a,b,c)$ is an open subgroup of $\widehat{P}$, in particular of finite index; our Theorem C is more refined, giving effective control over the closure of the image.

\subsection*{Applications}

The construction and analysis of these curves has several interesting applications.  Combining Theorems A and B, we see that the branched cover $X(a,b,c;\frakp) \to \PP^1$ realizes the group $\PSL_2(\F_\frakp)$ or $\PGL_2(\F_\frakp)$ regularly over the field $M(X,f,G)$, a small extension of a totally real abelian number field.  (See Malle and Matzat \cite{MalleMatzat}, Serre \cite[Chapters 7--8]{TGT}, and Volklein \cite{Volklein} for more information and groups realized regularly by rigidity and other methods.)

Moreover, the branched covers $X(a,b,c;\frakp) \to X(a,b,c)$ have applications in the Diophantine study of generalized Fermat equations.  When $c=\infty$, Darmon \cite{Darmon} has constructed a family of hypergeometric abelian varieties associated to the triangle group $\Deltabar(a,b,c)$.  The analogous construction when $c \neq \infty$ we believe will likewise have important arithmetic applications.  (See also work of Tyszkowska \cite{Tysz}, who studies the fixed points of a particular symmetry of $\PSL_2(\F_p)$-Galois \Belyi\ curves.)

Finally, it is natural to consider applications to the arithmetic theory of elliptic curves.  Every elliptic curve $E$ over $\Q$ is uniformized by a modular curve $X_0(N) \to E$, and the theory of Heegner points govern facets of the arithmetic of $E$: in particular, it controls the rank of $E(\Q)$ when this rank is at most $1$.  By analogy, we are led to consider those elliptic curves over a totally real field that are uniformized by a curve $X(a,b,c;\frakp)$---there is some evidence  \cite{JV:semiarith} that the images of CM points generate subgroups of rank at least $2$.

\subsection*{Organization}

The paper is organized as follows.  In Sections 1--3, we introduce triangle groups, \Belyi\ maps, Galois \Belyi\ curves, and fields of moduli.  In Section 4, we investigate in detail a construction of Takeuchi, later explored by Cohen and Wolfart, which realizes the curves associated to triangle groups as subvarieties of quaternionic Shimura varieties, and from this modular embedding we define congruence subgroups of triangle groups.  We next introduce in Section 5 the theory of weak rigidity which provides the statement of Galois descent we will employ.  In Section 6, we set up the basic theory of $\PSL_2(\F_q)$, and in Section 7 we recall Macbeath's theory of two-generated subgroups of $\SL_2(\F_q)$.  In Section 8, we put the pieces together and prove Theorems A, B, and C.  We conclude in Section 9 with several explicit examples.

The authors would like to thank Henri Darmon, Richard Foote, David Harbater, Hilaf Hasson, Robert Kucharczyk, Jennifer Paulhus, Jeroen Sijsling, and J\"urgen Wolfart for helpful discussions, as well as Noam Elkies for his valuable comments and encouragement.  The second author was supported by an NSF CAREER Award (DMS-1151047).

\section{Triangle groups}

In this section, we review the basic theory of triangle groups.  We refer to Magnus \cite[Chapter II]{Magnus} and Ratcliffe \cite[\S 7.2]{Ratcliffe} for further reading.

Let $a,b,c \in \Z_{\geq 2} \cup \{\infty\}$ satisfy $a \leq b \leq c$.  We say that the triple $(a,b,c)$ is \defi{spherical}, \defi{Euclidean}, or \defi{hyperbolic} according as the quantity
\[ \chi(a,b,c)=1-\frac{1}{a}-\frac{1}{b}-\frac{1}{c} \]
is negative, zero, or positive.  The spherical triples are $(2,3,3)$, $(2,3,4)$, $(2,3,5)$, and $(2,2,c)$ with $c \in \Z_{\geq 2}$.  The Euclidean triples are $(2,2,\infty)$, $(2,4,4)$, $(2,3,6)$, and $(3,3,3)$.  All other triples are hyperbolic.

We associate to a triple $(a,b,c)$ the \defi{extended triangle group} $\Delta=\Delta(a,b,c)$, the group generated by elements $-1,\delta_a,\delta_b,\delta_c$, with $-1$ central in $\Delta$, subject to the relations $(-1)^2=1$ and
\begin{equation}
\delta_a^a=\delta_b^b=\delta_c^c=\delta_a\delta_b\delta_c=-1;
\end{equation}
by convention we let $\delta_\infty^\infty=-1$.  We define the quotient
\[ \Deltabar=\Deltabar(a,b,c)=\Delta(a,b,c)/\{\pm 1\} \]
and call $\Deltabar$ a \defi{triangle group}.  We denote by $\deltabar$ the image of $\delta \in \Delta(a,b,c)$ in $\Deltabar(a,b,c)$.

\begin{rmk} \label{reorderremark}
Reordering generators permits our assumption that $a \leq b \leq c$ without loss of generality.  Indeed, the defining condition $\delta_a \delta_b \delta_c = -1$ is invariant under cyclic permutations so $\Delta(a,b,c) \cong \Delta(b,c,a) \cong \Delta(c,a,b)$, and similarly the map which sends a generator to its inverse gives an isomorphism $\Delta(a,b,c) \cong \Delta(c,b,a)$.  The same is true for the quotients $\overline{\Delta}(a,b,c)$.
\end{rmk}

The triangle groups $\Deltabar(a,b,c)$ with $(a,b,c)$ earn their name from the following geometric interpretation.  Associated to $\Deltabar$ is a triangle $T$ with angles $\pi/a$, $\pi/b$, and $\pi/c$ on the Riemann sphere, the Euclidean plane, or the hyperbolic plane according as the triple is spherical, Euclidean, or hyperbolic, where by convention we let $1/\infty=0$.  (The case $(a,b,c)=(2,2,\infty)$ is admittedly a bit weird; one must understand the term \emph{triangle} generously in this case.)  The group of isometries generated by reflections $\overline{\tau}_a,\overline{\tau}_b,\overline{\tau}_c$ in the three sides of the triangle $T$ is a discrete group with $T$ itself as a fundamental domain.  The subgroup of orientation-preserving isometries is generated by the elements $\deltabar_a=\overline{\tau}_b\overline{\tau}_c$, $\deltabar_b=\overline{\tau}_c\overline{\tau}_a$, and $\deltabar_c=\overline{\tau}_a\overline{\tau}_b$ and these elements generate a group isomorphic to $\Deltabar(a,b,c)$.  A fundamental domain for $\Deltabar(a,b,c)$ is obtained by reflecting the triangle $T$ in one of its sides.  The sides of this fundamental domain are identified by the elements $\deltabar_a$, $\deltabar_b$, and $\deltabar_c$, and consequently the quotient space is a Riemann surface of genus zero.  This surface is compact if and only if $a,b,c \neq \infty$ (i.e., $c \neq \infty$ since $a \leq b \leq c$).  We analogously classify the groups $\Deltabar(a,b,c)$ as spherical, Euclidean, or hyperbolic.  We make the convention $\Z/\infty \Z = \Z$.

\begin{exm}
For all $a,b \geq 2$, $\Deltabar(a,b,\infty)$ is canonically isomorphic to the free product $\Z/a\Z * \Z/b\Z$.  This group is Euclidean when $a = b = 2$
and otherwise hyperbolic.
\begin{enumerate}
\item The group $\Deltabar(2,2,\infty) = \Z/2\Z * \Z/2\Z$
can be geometrically realized as the group of isometries of the Euclidean plane generated by reflections through two distinct, parallel lines.  This yields
the alternate presentation
\[\Deltabar(2,2,\infty)  \cong \langle \sigma,\tau \mid
\sigma^2 = 1, \ \sigma \tau \sigma^{-1} = \tau^{-1} \rangle. \]
The group $\Deltabar(2,2,\infty)$ is sometimes called the \defi{infinite dihedral group}.
\item We have $\Delta(2,3,\infty) \cong \Z/4\Z *_{\Z/2\Z} \Z/6\Z \cong \SL_2(\Z)$.  It follows
that $\Deltabar(2,3,\infty) = \Z/2\Z * \Z/3\Z \cong \PSL_2(\Z)$.
\item The group $\Deltabar(\infty,\infty,\infty) = \Z * \Z$ is free
on two generators.  We have $\Deltabar(\infty,\infty,\infty) \cong \Ker(\PSL_2(\Z) \rightarrow \PSL_2(\Z/2\Z))$.
\item For $n \in \Z_{\geq 2}$, the groups $\Delta(2,n,\infty) \cong \Z/2\Z * \Z/n\Z$ are called \defi{Hecke groups} \cite{Hecke}.
\end{enumerate}
\end{exm}

\begin{exm}
\label{exm:sphericaltriang}
The spherical triangle groups are all finite groups: we have $\Deltabar(2,2,c) \cong D_{2c}$, the dihedral group on $2c$ elements, and
\[ \Deltabar(2,3,3) \cong A_4, \quad  \Deltabar(2,3,4) \cong S_4, \quad \Deltabar(2,3,5) \cong A_5. \]
\end{exm}

We have the exact sequence
\begin{equation} \label{deltabarab}
1 \to [\Deltabar,\Deltabar] \to \Deltabar \to \Deltabar^{\textup{ab}} \to 1
\end{equation}
where $[\Deltabar,\Deltabar]$ denotes the commutator subgroup.  If $c \neq \infty$, then $\Deltabar^{\textup{ab}}=\Deltabar/[\Deltabar,\Deltabar]$ is isomorphic to the quotient of $\Z/a\Z \times \Z/b\Z$ by the cyclic subgroup generated by $(c,c)$; when $c=\infty$, we have $\Deltabar^{\textup{ab}} \cong \Z/a\Z \times \Z/b\Z$.  Thus, the group $\Deltabar$ is perfect (i.e.\ $\Deltabar^{\textup{ab}}=\{1\}$) if and only if $a,b,c$ are relatively prime in pairs.  We have $[\Deltabar,\Deltabar] \cong \Z$ for $(a,b,c)=(2,\infty,\infty)$, whereas for the other
Euclidean triples we have $[\Deltabar,\Deltabar] \cong \Z^2$ \cite[\S II.4]{Magnus}.  In particular, the Euclidean triangle groups are infinite and nonabelian, but solvable.

From now on, suppose $(a,b,c)$ is hyperbolic.  Then by the previous paragraph we can realize $\Deltabar=\Deltabar(a,b,c) \hookrightarrow \PSL_2(\R)$ as a Fuchsian group, a discrete subgroup of orientation-preserving isometries of the upper half-plane $\calH$.   Let $\calH^{(*)}$ denote $\calH$ together
with the cusps of $\Deltabar(a,b,c)$: this is the number of instances of $\infty$ among $a,b,c$. We write $X(a,b,c)=\Deltabar(a,b,c) \backslash \calH^{(*)} \cong \PP_\C^1$ for the quotient space.

We lift this embedding to $\SL_2(\R)$ as follows.  Suppose that $b < \infty$: this excludes the cases $(a,\infty,\infty)$ and $(\infty,\infty,\infty)$, whose associated groups are commensurable with $\SL_2(\Z)$ and can be analyzed after making appropriate modifications.  Then Takeuchi \cite[Proposition 1]{Takeuchi} has shown that there exists an embedding
\[ \Delta(a,b,c) \hookrightarrow \SL_2(\R) \]
which is unique up to conjugacy in $\SL_2(\R)$.  In fact, this embedding can be made explicit as follows \cite{Petersson}.  As in the introduction, for $s \in \Z_{\geq 2}$, let $\zeta_s=\exp(2\pi i/s)$ and
\begin{equation} \label{lambdak}
\lambda_s=\zeta_{s}+\frac{1}{\zeta_{s}}=2\cos\left(\frac{2\pi}{s}\right) \text{\ \ and\ \ }\mu_s = 2\sin\left(\frac{2\pi}{s}\right) = -i\left(\zeta_{s}-\frac{1}{\zeta_{s}}\right)
\end{equation}
where by convention $\zeta_\infty=1$, $\lambda_\infty=2$, and $\mu_\infty=0$.

Then we have a map
\begin{equation} \label{explicitmap}
\begin{aligned}
\Delta(a,b,c) &\hookrightarrow \SL_2(\R) \\
\delta_a &\mapsto \frac{1}{2} \begin{pmatrix} \lambda_{2a} & \mu_{2a} \\ -\mu_{2a} & \lambda_{2a} \end{pmatrix} 
\\ \delta_b &\mapsto \frac{1}{2}
\begin{pmatrix} \lambda_{2b} & t\mu_{2b} \\ -\mu_{2b}/t & \lambda_{2b} \end{pmatrix} 
\end{aligned}
\end{equation}
where
\[ t+1/t=2\frac{\lambda_{2a}\lambda_{2b}+2\lambda_{2c}}{\mu_{2a}\mu_{2b}}. \]
The embedding (\ref{explicitmap}) then also gives rise to an explicit embedding $\Deltabar(a,b,c) \hookrightarrow \PSL_2(\R)$.

A triangle group $\Deltabar=\Deltabar(a,b,c)$ is \defi{maximal} (we also say that the triple $(a,b,c)$ is \defi{maximal}) if $\Deltabar$ cannot be properly embedded in any Fuchsian group.  By a result of Singerman \cite{Singerman} (see also Greenberg \cite[Theorem 3B]{Greenberg}), any Fuchsian group containing
 $\Deltabar(a,b,c)$ is itself a triangle group.  All inclusion relations between triangle groups can be generated (by concatenation) from the relations \cite[(2)]{GirondoWolfart}
\begin{equation} \label{fam2}
\begin{array}{cc}
\Deltabar(2,7,7) \leq_{9} \Deltabar(2,3,7)  \quad & \Deltabar(3,8,8) \leq_{10} \Deltabar(2,3,8) \\
\Deltabar(4,4,5) \leq_{6} \Deltabar(2,4,5) \quad & \Deltabar(3,3,7) \leq_{8} \Deltabar(2,3,7)
\end{array}
\end{equation}
or one of the families
\begin{equation} \label{fam1}
\begin{array}{cc}
\Deltabar(a,a,a) \trianglelefteq_3 \Deltabar(3,3,a) \quad & \Deltabar(a,a,c) \trianglelefteq_2 \Deltabar(2,a,2c) \\
\Deltabar(2,b,2b) \leq_3 \Deltabar(2,3,2b) \quad & \Deltabar(3,b,3b) \leq_4 \Deltabar(2,3,3b).
\end{array}
\end{equation}
Here $H \leq_n G$ (resp. $H \trianglelefteq_n G$) means that
$H$ is an index $n$ subgroup of $G$ (resp. an index $n$ normal
subgroup of $G$).  Moreover, to avoid tedious proliferation of cases, we have in (\ref{fam1}) removed our assumption that $a \leq b \leq c$.  It follows from \eqref{fam2}--\eqref{fam1} that $\Deltabar(a,b,c)$ is maximal if and only if $(a,b,c)$ is not of the form
\[ (a,a,c), (a,b,b), (2,b,2b), \text{ or } (3,b,3b) \]
with again $a,b,c \in \Z_{\geq 2} \cup \{\infty\}$ not necessarily in increasing order.

A Fuchsian group $\Gamma$ is \defi{arithmetic} \cite{Borelarith} if there exists a quaternion algebra $B$ over a totally real field $F$ that is unramified at precisely one real place of $F$ such that $\Gamma$ is commensurable with the image of the units of reduced norm $1$ in an order $\calO \subseteq B$.  Takeuchi \cite[Theorem 3]{Takeuchi} has enumerated the arithmetic triangle groups $\Deltabar(a,b,c)$: there are $85$ of them, falling into $19$ commensurability classes \cite[Table (1)]{Takeuchi2}.

\section{Galois \Belyi\ maps}

In this section, we discuss \Belyi\ maps and Galois \Belyi\ curves and we relate these curves to those uniformized by subgroups of triangle groups.

A \defi{branched cover} of curves over a field $k$ is a finite morphism of curves $f:X \to Y$ defined over $k$.  A \defi{\Belyi\ map} is a branched cover $f:X \to \PP^1$ over $\C$ which is unramified away from $\{0,1,\infty\}$.
An
\defi{isomorphism} of branched covers between $f$ and $f'$ is an isomorphism $h:X \xrightarrow{\sim} X'$ that respects the covers, i.e., such that $f=f'\circ h$.

\begin{rmk}
Let $f: X \rightarrow \PP^1$ be a morphism of degree $d > 1$.  By
Riemann-Hurwitz, $f$ is ramified over at least two points of $\PP^1$, and if
$f$ is ramified over exactly two points then $X \cong \PP^1$.  In the latter
case, after identifying $X$ with $\PP^1$ we may adjust the target by a linear
fractional transformation so as to have $f(z) = z^d$.
\end{rmk}

A branched cover that is a Galois (with Galois group $G$), i.e.\ a covering whose corresponding extension of function fields is Galois, is called a \defi{Galois branched cover}; if such a branched cover is further equipped with an isomorphism $i:G \xrightarrow{\sim} \Gal(f)=\Aut(X,f) \subseteq \Aut(X)$, it is called a \defi{$G$-Galois branched cover}.  Note the distinction between the two!  A curve $X$ that possesses a Galois \Belyi\ map is called a \defi{Galois \Belyi\ curve}.  An \defi{isomorphism} of $G$-Galois branched covers over $k$ is an isomorphism $h$ of branched covers that identifies $i$ with $i'$, i.e.,
\begin{center}
$h(i(g)x)=i'(g)h(x)$ for all $g \in G$ and $x \in X(\overline{k})$
\end{center}
where $\overline{k}$ is an algebraic closure of $k$.  (This distinction may seem irrelevant at first, but it is important if one wants to study properties not just the cover but also the Galois group of the branched cover.)  For a Galois branched cover $f: X \rightarrow \PP^1$, the ramification index of $P \in X(\C)$ depends only on $f(P)$, so we record these indices as a triple $(a_1,\dots,a_n)$ of integers $1 < a_1 \leq \dots \leq a_n$ and say that $(a_1,\dots,a_n)$ is the \defi{ramification type} of $f$.

\begin{rmk} \label{rmk:XautX}
If $X$ has genus at least $2$ and $X \to X/G$ is a $G$-Galois \Belyi\ map, then the quotient $X \to X/\!\Aut(X)$ is a $\Aut(X)$-Galois \Belyi\ map.
\end{rmk}

\begin{exm} \label{Wolfartgenus0}
The map
\begin{align*}
f:\PP^1 &\to \PP^1 \\
f(t) &= \frac{t^2(t+3)}{4}=1+\frac{(t-1)(t+2)^2}{4}
\end{align*}
is a \Belyi\ map, a branched cover ramified only over $0,1,\infty$, with ramification indices $(2,2,3)$.  In particular, $\PP^1$ is a Galois \Belyi\ curve.  The Galois closure of $f$ is a Galois \Belyi\ map $\PP^1 \to \PP^1$ with Galois group $S_3$ corresponding to the simplest spherical triangle group $\Deltabar(2,2,3)$: it is given by
\[ f(t)=\frac{27t^2(t-1)^2}{4(t^2-t+1)^3} \quad \text{with} \quad
 f(t)-1 = -\frac{(t-2)^2(2t-1)^2(t+1)^2}{4(t^2-t+1)^3}. \]
 It becomes an $S_3$-Galois \Belyi\ map when it is equipped with the isomorphism
\begin{align*}
S_3 &\xrightarrow{\sim} \Gal(f) \leq \Aut(\PP^1) \cong \PGL_2(\C) \\
(1\ 2) &\mapsto (t \mapsto 1-t) \leftrightarrow \begin{pmatrix} -1 & 1 \\ 0 & 1 \end{pmatrix} \\
(1\ 2\ 3) &\mapsto \left(t \mapsto \frac{1}{1-t}\right) \leftrightarrow \begin{pmatrix} 0 & 1 \\ -1 & 1 \end{pmatrix}.
\end{align*}
All examples of Galois \Belyi\ maps $\PP^1 \to \PP^1$ arise in this way from the spherical triangle groups, as in Example \ref{exm:sphericaltriang}.
\end{exm}


\begin{exm} \label{Wolfartgenus1}
We now consider Galois \Belyi\ maps $E \to \PP^1$ where $E$ is a curve of genus $1$ over $\C$.  There is no loss in assuming that $E$ has the structure of elliptic curve with neutral element $\infty \in E(\C)$.  The elliptic curves with extra automorphisms present candidates for such maps.

The curve $E:y^2=x^3-x$ with $j(E)=1728$ has $G=\Aut(E,\infty)$ cyclic of order $4$, and the quotient $x^2:E \to E/G \cong \PP^1$ yields a Galois \Belyi\ map of degree $4$ with ramification type $(2,4,4)$ by a direct computation.  This map as a $G$-Galois \Belyi\ map is minimally defined over $\Q(\sqrt{-1})$; the Belyi map itself is defined over $\Q$.

Next we consider the curve with $j(E)=0$ with $G=\Aut(E,\infty)$ cyclic of order $6$, from which we obtain two Galois \Belyi\ maps.  The first map is obtained by writing $E:y^2=x^3-1$ and taking the map $x^3:E \to E/G \cong \PP^1$, a Galois \Belyi\ map of degree $6$ with ramification type $(2,3,6)$.  The second is obtained by writing instead $E:y^2-y=x^3$ (isomorphically) and the unique subgroup $H < G$ of order $3$, corresponding to the map $y:E \to E/H \cong \PP^1$ with ramification type $(3,3,3)$.  These maps are minimally defined over $\Q(\sqrt{-3})$ as Galois \Belyi\ maps.  Indeed, the inclusions \eqref{fam1} imply an inclusion $\Deltabar(3,3,3) \trianglelefteq_2 \Deltabar(2,3,6)$, so the former is the composition of the latter together with the squaring map.

One obtains further Galois \Belyi\ maps by precomposing these with an isogeny $E \to E$.
\end{exm}

\begin{lem} \label{lem:Wolfartgenus1}
Up to isomorphism, the only Galois \Belyi\ maps $E \to \PP^1$ with $E$ a genus $1$ curve over $\C$ are of the form
\[ E \xrightarrow{\phi} E \xrightarrow{f} \PP^1 \]
where $\phi$ is an isogeny and $f$ is one of the three Galois \Belyi\ maps in Example \textup{\ref{Wolfartgenus1}}.  In particular, the only Galois \Belyi\ curves $E$ of genus $1$ have $j(E)=0,1728$.
\end{lem}

\begin{proof}
Let $E \to \PP^1$ be a Galois \Belyi\ map, where without loss of generality we may assume $E:y^2=f(x)$ is an elliptic curve in Weierstrass form with neutral element $\infty$.  We claim that $j(E)=0,1728$.  We always have $\Aut E = E(\C) \rtimes \Aut (E,\infty)$.  If $G < \Aut E$ is a finite subgroup, then $G' = G \cap E(\C) \trianglelefteq G$ and $G/G' \subseteq \Aut (E,\infty)$, so $E' = E/G'$ is an elliptic curve and $E/G \cong E/(G/G')$.  However, if $j(E) \neq 0,1728$, then $\Aut (E,\infty) = \{\pm 1\}$, so either $G = G'$ and $E/G$ is an elliptic curve, or $G=\pm G'$ and the map $E \rightarrow E/G' \xrightarrow{x} E/G \cong \PP^1$ is ramified at four points, the roots of $f(x)$ and $\infty$.
\end{proof}

In view of Examples \ref{Wolfartgenus0} and \ref{Wolfartgenus1} and Lemma \ref{lem:Wolfartgenus1}, from now on we may restrict our attention to Galois \Belyi\ maps $f:X \to \PP^1$ with $X$ of genus $g \geq 2$.  These curves can be characterized in several equivalent ways.

\begin{prop}[Wolfart {\cite{Wolfart97,Wolfart00}}] \label{wolfartequiv}
Let $X$ be a compact Riemann surface of genus $g \geq 2$.  Then the following are equivalent.
\begin{enumroman}
\item $X$ is a Galois \Belyi\ curve;
\item The map $X \to X/\!\Aut(X)$ is a \Belyi\ map;
\item There exists a finite index, torsion-free normal subgroup $\Gamma \trianglelefteq \Deltabar(a,b,c)$ with $a,b,c \in \Z_{\geq 2}$ and a complex uniformization $\Gamma \backslash \calH \xrightarrow{\sim} X$; and
\item There exists an open neighborhood $U$ of $[X]$ (with respect to the complex analytic topology) in the moduli space $\calM_g(\C)$ of curves of genus $g$ such that $\#\Aut(X) > \#\Aut(Y)$ for all $[Y] \in U \setminus \{[X]\}$.
\end{enumroman}
\end{prop}

\begin{rmk}
Proposition \ref{wolfartequiv} implies that Riemann surfaces uniformized by  subgroups of non-cocompact hyerperbolic triangle groups are also uniformized by subgroups of cocompact hyperbolic triangle groups.  More precisely: let $a',b' \in \Z_{\geq 2} \cup \{\infty\}$, $(a',b') \neq (2,2)$, and let $\Gamma' \subset \Deltabar(a',b',\infty)$ be a finite index subgroup (\emph{not} necessarily torsionfree).   Then $\Gamma' \backslash \calH^{(*)} \rightarrow
\Deltabar(a',b',\infty) \backslash \calH^{(*)}$ is a Galois \Belyi\ map, so by
Proposition \ref{wolfartequiv} there are $a,b,c \in \Z_{\geq 2}$ and a finite index, normal torsionfree subgroup $\Gamma \subset \Deltabar(a,b,c)$ such that
$\Gamma' \backslash \calH^{(*)} \cong \Gamma \backslash \calH$.  The case of
$\PSL_2(\F_q)$-Galois \Belyi\ curves uniformized by subgroups of Hecke triangle groups is treated in detail by Schmidt and Smith \cite[Prop. 4]{SchmidtSmith}.
\end{rmk}

By the Riemann-Hurwitz formula, if $X$ is a $G$-Galois \Belyi\ curve of type $(a,b,c)$, then $X$ has genus
\begin{equation} \label{Wolfartprops}
g(X)=1+\frac{\#G}{2}\left(1-\frac{1}{a}-\frac{1}{b}-\frac{1}{c}\right) = 1 - \frac{\#G}{2}\chi(a,b,c).
\end{equation}

\begin{rmk} \label{Hurwitz}
The function of $\#G$ in (\ref{Wolfartprops}) is maximized when $(a,b,c)=(2,3,7)$.  Combining this with Proposition \ref{wolfartequiv}(iv) we recover the Hurwitz bound
\[ \#\Aut(X) \leq 84(g(X)-1). \]
\end{rmk}

\begin{rmk} \label{onlyfinGaloisBelyi}
There are only finitely many Galois \Belyi\ curves of any given genus $g$.  By the Hurwitz bound (\ref{Hurwitz}), we can bound $\#G$ given $g \geq 2$, and for
fixed $g$ and $\#G$ there are only finitely many triples $(a,b,c)$ satisfying
(\ref{Wolfartprops}).  Each $\Deltabar(a,b,c)$ is finitely generated so has
only finitely many subgroups of index $\# G$.  From this, one can extract an explicit upper bound; using a more refined approach, Schlage-Puchta and Wolfart \cite[Thm. 1]{SchPuchWolf} showed that the number of  isomorphism classes of Galois \Belyi\ curves of genus at most $g$ grows like $g^{\log g}$.
\end{rmk}

\begin{rmk}
Wolfart \cite{Wolfart00} gives a complete list of all Galois \Belyi\ curves of genus $g=2,3,4$.  Further examples of Galois \Belyi\ curves can be found in the work of Shabat and Voevodsky \cite{SV}. See Table \ref{tablelowgenus} for the determination of all $\PSL_2(\F_q)$-Galois \Belyi\ curves with genus $g \leq 24$.
\end{rmk}

\begin{exm} \label{Galoisclosure}
Let $f:X \to \PP^1$ be a \Belyi\ map and let $g:Y \to \PP^1$ be its Galois closure.  Then $g$ is also a \Belyi\ map and hence $Y$ is a Galois \Belyi\ curve.  Note however that the genus of $Y$ may be much larger than the genus of $X$!
\end{exm}

Condition Proposition \ref{wolfartequiv}(iii) leads us to consider curves arising from finite index normal subgroups of the hyperbolic triangle groups $\Deltabar(a,b,c)$.  If $\Gammabar \subseteq \PSL_2(\R)$ is a Fuchsian group, write $X(\Gammabar)=\Gammabar \setminus \calH^{(*)}$.  If $X$ is a compact Riemann surface of genus $g \geq 2$ with uniformizing subgroup $\Gammabar \subseteq \PSL_2(\R)$, so that $X=X(\Gammabar)$, then $\Aut(X)=N(\Gammabar)/\Gammabar$, where $N(\Gammabar)$ is the normalizer of $\Gammabar$ in $\PSL_2(\R)$.  Moreover, the quotient $X \to X/\!\Aut(X)$, obtained from the map $X(\Gammabar) \to X(N(\Gammabar))$, is a Galois cover with Galois group $\Aut(X)$.
By the results of Section 1, if $\Gammabar \subseteq \Deltabar(a,b,c)$ is a finite index normal subgroup then $\Aut(X(\Gammabar))$ is of the form $\Deltabar'/\Gammabar$ with an inclusion $\Deltabar \subseteq \Deltabar'$ as in (\ref{fam2})--(\ref{fam1}); if $\Deltabar$ is maximal, then we have
\begin{equation} \label{ifmaxdelta}
\Aut(X(\Gammabar)) \cong \Deltabar(a,b,c)/\Gammabar.
\end{equation}

\section{Fields of moduli}

In this section, we briefly review the theory of fields of moduli and fields of definition.  See Coombes and Harbater \cite{CoombesHarbater} and K\"ock \cite{Kock} for more detail.

The \defi{field of moduli} $M(X)$ of a curve $X$ over $\C$ is the fixed field of the group
\[ \{\sigma \in \Aut(X) : X^{\sigma} \cong X\}. \]
In a similar way, we define the fields of moduli $M(X,f)$ of a \Belyi\ map $f:X \to \PP^1$ and $M(X,f,G)$ of a $G$-Galois \Belyi\ map.

Owing to a lack of rigidity, not every curve can be defined over its field of moduli.  However, in our situation we have the following lemma.

\begin{lem} \label{canbedefined}
Let $f:X \to \PP^1$ be a Galois \Belyi\ map.  Then $f$ is defined over its field of moduli $M(X,f)$.
More generally, let $X$ be a Galois \Belyi\ curve.  Then $X$ is defined over its field of moduli $M(X)$, and $M(X)=M(X,f)$ where $f:X \to X/\!\Aut(X) \cong \PP^1$.
\end{lem}

\begin{proof}
D\`ebes and Emsalem \cite{DebesEmsalem} remark that this lemma follows from results of Coombes and Harbater \cite{CoombesHarbater}.  The proof was written down by K\"ock \cite[Theorem 2.2]{Kock}: in fact, he shows that any Galois covering of curves $X \to \PP^1$ can be defined over the field of moduli of the cover (similarly defined), and the field of moduli of $X$ as a curve is equal to the field of moduli of the covering $X \to X/\!\Aut(X)$.  See also the proof in Wolfart \cite[Theorem 5]{WolfartABC}.
\end{proof}

Keeping track of the action of the automorphism group, we also have the following result.

\begin{lem} \label{lem:canbedefinedG}
Let $f:X \to \PP^1$ be a $G$-Galois \Belyi\ map.  Suppose that $C_{\Aut(X)}(G)=\{1\}$, i.e., the centralizer of $G$ in $\Aut(X)$ is trivial. Then $f$ and the action of $\Gal(f)\cong G$ can be defined over its field of moduli $M(X,f,G)$.
\end{lem}

\begin{proof}
By definition, an automorphism of $f$ as a $G$-Galois \Belyi\ map is given by $h \in \Aut(X)$ such that $h i(g) h^{-1} = i(g)$ for all $g \in G$, so under the hypothesis of the lemma, $f$ has no automorphisms.  Thus $f$ and $\Gal(f)$ can be defined over $M(X,f,G)$ by the criterion of Weil descent.
\end{proof}

\begin{rmk}
Let $X$ be a curve which can be defined over its field of moduli $F=M(X)$.  Then the set of $F$-isomorphism classes of models for $X$ over $F$ is in bijection with the Galois cohomology set $H^1(\Gal(\overline{F}/F), \Aut(X))$, where $\Aut(X)$ is equipped with the natural action of the absolute Galois group $\Gal(\overline{F}/F)$.  Similar statements are true more generally for the other objects considered here, including \Belyi\ maps and $G$-Galois \Belyi\ maps.
\end{rmk}

As a consequence of Lemma \ref{lem:canbedefinedG}, if $G \cong \Aut(X)$ and $G$ has trivial center $Z(G)=\{1\}$, then $f$ as a $G$-Galois \Belyi\ map can be defined over $M(X,f,G)$.  Under this hypothesis, if $K=M(X,f,G)$, then by definition the group $G$ occurs as a Galois group over $K(t)$, and in particular applying Hilbert's irreducibility theorem \cite[Chapter 3]{TGT} we find that $G$ occurs infinitely often as a Galois group over $K$.

\begin{exm} \label{XpoverQ}
Let $p$ be prime and let $X(p)_{/\C} =\Gamma(p) \backslash \calH^*$ be the classical modular curve, parametrizing (generalized) elliptic curves $E$ equipped with a basis of $E[p]$ which is symplectic with respect to the Weil pairing.  Then $\Aut(X(p)) \supseteq G=\PSL_2(\F_p)$, and the quotient map $j:X \to X/G \cong \PP^1$, corresponding to the inclusion $\Gamma(p) \subseteq \PSL_2(\Z)$, is ramified over $j=0,1728,\infty$ with indices $2,3,p$, so $X(p)$ is a Galois \Belyi\ curve.

For $p \leq 5$, the curve $X(p)$ has genus $0$ and thus
$\Aut X(p) = \PGL_2(\C)$.  For $p \geq 7$, the curve $X(p)$
has genus at least three (the curve $X(7)$ has genus $3$ and
is considered in more detail in the following example), so
$\Aut X(p)$ is a finite group containing $\PSL_2(\F_p)$.  In
fact we have $\Aut X(p) = \PSL_2(\F_p)$, as was shown by Mazur, following Serre \cite[p. 255]{Mazur}.  Later we will recover
this fact as a special case of a more general result.

The field of moduli of $j:X \to \PP^1$ is $\Q$, and indeed this map (and hence $X$) admits a canonical model over $\Q$ \cite{KatzMazur}.  This model is not unique, since the set $H^1(\Q, \Aut(X))$ is infinite: in fact, every isomorphism class of Galois modules $E[p]$ with $E$ an elliptic curve gives a different element in this set.

For $p> 2$, let $p^* = (-1)^{(p-1)/2}$ (so $\Q(\sqrt{p^*})$ is
the unique quadratic subfield of $\Q(\zeta_p)$).  The field of moduli of the $\PSL_2(\F_p)$-Galois \Belyi\ map $j$ is $\Q(\sqrt{p^*})$ when $p > 2$ and $\Q$ when $p = 2$, and in each case the field of moduli is a field of definition \cite[pp. 108-109]{Shih}.  Indeed, this follows from Weil descent when $p \geq 7$ and can be seen directly when $p=2,3,5$ as these correspond to spherical triples $(2,3,p)$ (cf.\ Example \ref{Wolfartgenus0}).
\end{exm}

\begin{exm} \label{Kleincurve}
The Klein quartic curve  \cite{Elkies}
 \[ X^3 Y + Y^3 Z + Z^3 X = 0 \]
 has field of definition equal to its field of moduli, which is $\Q$, and all elements of $\Aut(X)$ can be defined over $\Q(\sqrt{-7}) =
 \Q(\sqrt{7^*})$.  Although the Klein quartic is isomorphic to $X(7)$ over $\Qbar$, as remarked by Livn\'e, the Katz-Mazur canonical model of $X(7)$ agrees with the Klein quartic only over $\Q(\sqrt{-3})$.  The issue here concerns the fields of definition of the special points giving rise to the canonical model.  We do not go further into this issue here, but for more on this in the case of genus $1$, see work of Sijsling \cite{Sijsling}.
\end{exm}

\begin{rmk} \label{allWolfart}
We consider again Remark \ref{Galoisclosure}.  If the field of moduli of a \Belyi\ map $f:X \to \PP^1$ is $F$ then the field of moduli of its Galois closure $g:Y \to \PP^1$ as a \Belyi\ map contains $F$.  Consequently, let $F$ be a number field and let $X$ be an elliptic curve such that $\Q(j(X))=F$.  Then $X$ admits a \Belyi\ map defined over $F$.  The Galois closure $g:Y \to \PP^1$ therefore has field of moduli containing $F$, and so  for any number field $F$, there exists a $G$-Galois \Belyi\ map such that any field of definition of this map contains $F$.  Note that from Lemma \ref{lem:Wolfartgenus1} that outside of a handful of cases, the associated Galois \Belyi\ curve $Y$ has genus $g(X) \geq 2$.  This shows that $\Gal(\Qbar/\Q)$ acts faithfully on the set of isomorphism classes of $G$-Galois \Belyi\ curves.  However if $X \to \PP^1$ is a $G$-Galois \Belyi\ map and $H \leq G$ is a subgroup, then the field of moduli of $X \to X/H$ can be smaller
than the $M(X,f,G)$.

Nevertheless, Gonz\'alez-Diez and Jaikin-Zapirain \cite{GJZ} have recently shown that $\Gal(\Qbar/\Q)$ acts faithfully on the set of Galois \Belyi\ curves.
\end{rmk}

In view of Remark \ref{allWolfart}, we restrict our attention from the general setup to the special class of $G$-Galois \Belyi\ curves $X$ where $G=\PSL_2(\F_q)$ or $\PGL_2(\F_q)$.

\section{Congruence subgroups of triangle groups} \label{quatconst}

In this section, we associate a quaternion algebra over a totally real field to a triangle group following Takeuchi \cite{Takeuchi0}.  This idea was also pursued by Cohen and Wolfart \cite{CohenWolfart} with an eye toward results in transcendence theory, and further elaborated by Cohen, Itzykson, and Wolfart \cite{CoItWol}.  Here, we use this embedding to construct congruence subgroups of $\Delta$.  We refer to Vign\'eras \cite{Vigneras} for the facts we will use about quaternion algebras and Katok \cite{Katok} as a reference on Fuchsian groups.

Let $\Gamma \subseteq \SL_2(\R)$ be a subgroup such that $\Gammabar=\Gamma/\{\pm 1\} \subseteq \PSL_2(\R)$ has finite coarea, so in particular is $\Gamma$ is finitely generated.  Let
\[ F=\Q(\tr \Gamma)=\Q(\tr \gamma)_{\gamma \in \Gamma} \]
be the \defi{trace field} of $\Gamma$.  Then $F$ is a finitely generated extension of $\Q$.

Suppose further that $F$ is a number field, so $F$ has finite degree over $\Q$, and let $\Z_{F}$ be its ring of integers.  Let $F[\Gamma]$ be the $F$-vector space generated by $\Gamma$ in $\M_2(\R)$, and let $\Z_{F}[\Gamma]$ denote the $\Z_F$-submodule of $F[\Gamma]$ generated by $\Gamma$.  By work of Takeuchi \cite[Propositions 2--3]{Takeuchi00}, the ring $F[\Gamma]$ is a quaternion algebra over $F$.  If further $\tr(\Gamma) \subseteq \Z_F$, then $\Z_F[\Gamma]$ is an order in $F[\Gamma]$.

\begin{rmk}
 Schaller and Wolfart \cite{SchallerWolfart} call a Fuchsian group $\Gamma$ \defi{semi-arithmetic} if its trace field $F=\Q(\tr\Gamma)$ is a totally real number field and $\{ \tr \gamma^2 : \gamma \in \Gamma\}$ is contained in the ring of integers of $F$.  They ask if all semi-arithmetic groups are either arithmetic or subgroups of triangle groups; this conjecture remains open.  This is implied by a conjecture of Chudnovsky and Chudnovsky \cite[Section 7]{ChudChud}.  The Chudnovskys' conjecture is false if the group is not cocompact---this is implicit in work of McMullen and made explicit in work of Bouw and M\"oller \cite{BouwMoeller,BouwMoeller2}---but may still be true in the compact case.  See also work of Ricker \cite{Ricker}.
\end{rmk}

Let $(a,b,c)$ be a hyperbolic triple with $2 \leq a \leq b \leq c \leq \infty$.  As in \S 1, associated to the triple $(a,b,c)$ is the triangle group $\Delta(a,b,c) \subseteq \SL_2(\R)$ with $\Delta(a,b,c)/\{\pm 1\} \cong \Deltabar(a,b,c) \subseteq \PSL_2(\R)$.  Let $F=\Q(\tr \Delta(a,b,c))$ be the trace field of $\Delta(a,b,c)$.  The generating elements $\delta_s \in \Delta(a,b,c)$ for $s=a,b,c$ satisfy the quadratic equations
\[ \delta_s^2-\lambda_{2s}\delta_s + 1 = 0 \]
in $B$ where $\lambda_{2s}$ is defined in (\ref{lambdak}).


\begin{lem}[{\cite[Lemma 2]{Takeuchi}}] \label{lemgentr}
Let $\Gamma \subseteq \SL_2(\R)$.  If $\gamma_1,\dots,\gamma_r$ generate $\Gamma$, then $\Q(\tr \Gamma)$ is generated by $\tr(\gamma_{i_1}\cdots \gamma_{i_s})$ for $\{i_1,\dots,i_s\} \subseteq \{1,\dots,r\}$.
\end{lem}

\begin{cor}
If $\Gamma$ is generated by finitely many elliptic elements, then its trace field $F=\Q(\tr \Gamma)$ is a totally real number field.
\end{cor}


By Lemma \ref{lemgentr}, we deduce
\[ F=\Q(\tr \Delta(a,b,c)) = \Q(\lambda_{2a},\lambda_{2b},\lambda_{2c}). \]
Taking traces in the equation
\[ \delta_a\delta_b=-\delta_c^{-1}=\delta_c-\lambda_{2c}, \]
yields
\[ -\tr(\delta_c^{-1}) = -\lambda_{2c} = \tr(\delta_a\delta_b)=\delta_a \delta_b + (\lambda_{2b}-\delta_b)(\lambda_{2a}-\delta_a). \]
Also we have
\begin{equation} \label{gammaagammab}
\delta_a\delta_b+\delta_b\delta_a=\lambda_{2b}\delta_a+\lambda_{2a}\delta_b - \lambda_{2c}-\lambda_{2a}\lambda_{2b}.
\end{equation}
Together with the cyclic permutations of these equations, we conclude that the elements $1,\delta_a,\delta_b,\delta_c$
form a $\Z_F$-basis for the order $\calO=\Z_{F}[\Delta] \subseteq B=F[\Delta]$ (see also Takeuchi \cite[Proposition 3]{Takeuchi}).

\begin{lem}
The reduced discriminant of $\calO$ is a principal $\Z_{F}$-ideal generated by
\[ \beta=\lambda_{2a}^2+\lambda_{2b}^2+\lambda_{2c}^2+\lambda_{2a}\lambda_{2b}\lambda_{2c}-4
= \lambda_a + \lambda_b + \lambda_c + \lambda_{2a}\lambda_{2b}\lambda_{2c} + 2. \]
\end{lem}

\begin{proof}
Let $\frakd$ be the discriminant of $\calO$.  Then we calculate from the definition that
\[ \frakd^2=\det
\begin{pmatrix}
2 & \lambda_{2a} & \lambda_{2b} & \lambda_{2c} \\
\lambda_{2a} & \lambda_{2a}^2-2 & -\lambda_{2c} & -\lambda_{2b} \\
\lambda_{2b} & -\lambda_{2c} & \lambda_{2b}^2-2 & -\lambda_{2a} \\
\lambda_{2c} & -\lambda_{2b} & -\lambda_{2a} & \lambda_{2c}^2-2
\end{pmatrix} \Z_{F}
= \beta^2 \Z_{F}. \]
Alternatively, we compute a generator for $\frakd$ using the scalar triple product and (\ref{gammaagammab}) as
\begin{align*}
\tr([\delta_a,\delta_b]\delta_c) &= \tr((\delta_a\delta_b-\delta_b\delta_a)\delta_c)
= \tr(2\delta_a\delta_b - (\lambda_{2b}\delta_a+\lambda_{2a}\delta_b-\lambda_{2c}-\lambda_{2a}\lambda_{2b})\delta_c) \\
&= -4 - \lambda_{2b}\tr(\delta_a\delta_c) - \lambda_{2a}\tr(\delta_b\delta_c)+\lambda_{2c}^2 + \lambda_{2a}\lambda_{2b}\lambda_{2c} = \beta
\end{align*}
since $\delta_a\delta_c=-\delta_b^{-1}$ and $\delta_b\delta_c=-\delta_a^{-1}$.
\end{proof}

\begin{lem} \label{pnmid}
If $\frakP$ is a prime of $\Z_{F}$ with $\frakP \nmid 2abc$, then $\frakP \nmid \beta$.  If further $(a,b,c)$ is not of the form $(mk,m(k+1),mk(k+1))$ with $k,m \in \Z$, then $\frakP \nmid \beta$ for all $\frakP \nmid abc$.
\end{lem}

\begin{proof}
Let $\frakP$ be a prime of $F$ such that $\frakP \nmid abc$.  We have the following identity in the field $\Q(\zeta_{2a},\zeta_{2b},\zeta_{2c})=K$:
\begin{equation} \label{betafactor}
\beta=\left(\frac{\zeta_{2b}\zeta_{2c}}{\zeta_{2a}}+1\right)\left(\frac{\zeta_{2a}\zeta_{2c}}{\zeta_{2b}}+1\right)
\left(\frac{\zeta_{2a}\zeta_{2b}}{\zeta_{2c}}+1\right)\left(\frac{1}{\zeta_{2a}\zeta_{2b}\zeta_{2c}}+1\right).
\end{equation}
Let $\frakP_K$ be a prime above $\frakP$ in $K$ and suppose that $\frakP_K \mid \beta$.  Then $\frakP_K$ divides one of the factors in (\ref{betafactor}).

First, suppose that $\frakP_K \mid (\zeta_{2b}\zeta_{2c}\zeta_{2a}^{-1}+1)$, i.e., we have $\zeta_{2b}\zeta_{2c} \equiv -\zeta_{2a} \pmod{\frakP_K}$.  Suppose that $\frakP_K \nmid 2abc$.  Then the map $(\Z_K^\times)\sbtors \to \F_{\frakP_K}^\times$ is injective.  Hence $\zeta_{2b}\zeta_{2c}=-\zeta_{2a} \in K$.  But then embedding $K \hookrightarrow \C$ by $\zeta_s \mapsto e^{2\pi i/s}$ in the usual way, this equality would then read
\begin{equation} \label{abc1}
\frac{1}{b}+\frac{1}{c}=1+\frac{1}{a} \in \Q/2\Z.
\end{equation}
However, we have
\[ 0\leq\frac{1}{b}+\frac{1}{c}\leq 1<1+\frac{1}{a} < 2 \]
for any $a,b,c \in \Z_{\geq 2} \cup \{\infty\}$ when $a \neq \infty$, a contradiction, and when $a=\infty$ we have $b=c=\infty$ which again contradicts (\ref{abc1}).

Now suppose $\frakP_K \mid 2$ but still $\frakP_K \nmid abc$.  Then $\ker((\Z_K^\times)\sbtors \to \F_{\frakP}^\times)=\{\pm 1\}$, so instead we have the equation $\zeta_{2b}\zeta_{2c}=\pm \zeta_{2a} \in K$.  Arguing as above, it is enough to consider the equation with the $+$-sign, which is equivalent to
\[ \frac{1}{b}+\frac{1}{c}=\frac{1}{a}. \]
Looking at this equation under a common denominator we find that $b \mid c$, say $c=kb$.  Substituting this back in we find that $(k+1) \mid b$ so $b=m(k+1)$ and hence $a=km$ and $c=mk(k+1)$, and in this case we indeed have equality.

The case where $\frakP_K$ divides the middle two factors is similar.  The case where $\frakP_K$ divides the final factor follows from the impossibility of
\[ 0 = 1 + \frac{1}{a}+\frac{1}{b}+\frac{1}{c} \in \Q/2\Z \]
since $(a,b,c)$ is hyperbolic.
\end{proof}

We have by definition an embedding
\[ \Delta \hookrightarrow \calO_1^\times=\{\gamma \in \calO : \nrd(\gamma)=1\} \]
(where $\nrd$ denotes the reduced norm) and hence an embedding
\begin{equation} \label{basicembed}
\Deltabar = \Delta/\{\pm 1\} \hookrightarrow \calO_1^\times/\{\pm 1\}.
\end{equation}

In fact, the image of this map arises from a quaternion algebra over a smaller field, as follows.  Let $\Delta^{(2)}$ denote the subgroup of $\Delta$ generated by $-1$ and $\gamma^2$ for $\gamma \in \Delta$.  Then $\Delta^{(2)}$ is a normal subgroup of $\Delta$, and the quotient $\Delta/\Delta^{(2)}$ is an elementary abelian $2$-group.
We have an embedding
\[ \Delta^{(2)}/\{\pm 1\} \hookrightarrow \Delta/\{\pm 1\} = \Deltabar. \]

Recall the exact sequence (\ref{deltabarab}):
\[ 1 \to [\Deltabar,\Deltabar] \to \Deltabar \to \Deltabar^{\textup{ab}} \to 1. \]
Here, $\Deltabar^{\textup{ab}}$ is the quotient of $\Z/a\Z \times \Z/b\Z \times \Z/c\Z$ by the subgroup $(1,1,1)$.  We obtain $\Delta^{(2)} \supseteq [\Deltabar,\Deltabar]$ as the kernel of the (further) maximal elementary $2$-quotient of $\Deltabar^{\textup{ab}}$.  It follows that the quotient $\Delta/\Delta^{(2)}$ is generated by the elements $\delta_s$ for $s \in \{a,b,c\}$ such that either $s=\infty$ or $s$ is even, and
\begin{equation} \label{take124}
\Delta/\Delta^{(2)} \cong
\begin{cases}
\{0\}, & \text{if at least two of $a,b,c$ are odd}; \\
\Z/2\Z, & \text{if exactly one of $a,b,c$ is odd}; \\
(\Z/2\Z)^2, & \text{if all of $a,b,c$ are even or $\infty$}.
\end{cases}
\end{equation}
(See also Takeuchi \cite[Proposition 5]{Takeuchi}.)

Consequently, $\Delta^{(2)}$ is the normal closure of the set $\{-1, \delta_a^2,\delta_b^2,\delta_c^2\}$ in $\Delta$.  A modification of the proof of Lemma \ref{lemgentr} shows that the trace field of $\Delta^{(2)}$ can be computed on these generators (trace is invariant under conjugation).  We have
\[ \tr \delta_s^2 = \tr(\lambda_{2s} \delta_s-1)=\lambda_{2s}^2-2=\lambda_s-2 \]
for $s \in \{a,b,c\}$ and similarly
\[ \tr(\delta_a^2\delta_b^2)=\tr((\lambda_{2a}\delta_a-1)(\lambda_{2b}\delta_b-1))=\lambda_{2a}\lambda_{2b}\lambda_{2c} - \lambda_{2b}^2 - \lambda_{2a}^2 + 2 \]
and
\[ \tr(\delta_a^2 \delta_b^2 \delta_c^2) = \tr((\lambda_{2a} \delta_a-1)(\lambda_{2b} \delta_b-1)(\lambda_{2c} \delta_c-1)) = \lambda_{2a}^2 + \lambda_{2b}^2 + \lambda_{2c}^2 + \lambda_{2a}\lambda_{2b}\lambda_{2c} - 2; \]
from these we conclude that the trace field of $\Delta^{(2)}$ is equal to
\begin{equation} \label{deffabc}
E=F(a,b,c)=\Q(\lambda_{2a}^2,\lambda_{2b}^2,\lambda_{2c}^2,\lambda_{2a}\lambda_{2b}\lambda_{2c}) = \Q(\lambda_a,\lambda_b,\lambda_c,\lambda_{2a}\lambda_{2b}\lambda_{2c}).
\end{equation}
(See also Takeuchi \cite[Propositions 4--5]{Takeuchi}.)

\begin{exm}
The Hecke triangle groups $\Delta(2,n,\infty)$ for $n \geq 3$ have trace field $F=\Q(\lambda_{2n})$ whereas the corresponding groups $\Delta^{(2)}$ have trace field $E=\Q(\lambda_n)$, which is strictly contained in $F$ if and only if $n$ is even.
\end{exm}

Let $\Lambda=\Z_{E}[\Delta^{(2)}] \subseteq A = E[\Delta^{(2)}]$ be the order and quaternion algebra associated to $\Delta^{(2)}$.  By construction we have
\begin{equation} \label{delta2barembed}
\Delta^{(2)}/\{\pm 1\} \hookrightarrow \Lambda^{\times}_1/\{\pm 1\}.
\end{equation}

We then have the following fundamental result.

\begin{prop} \label{defnembedyeah}
The image of the natural homomorphism
\[ \Deltabar \hookrightarrow \frac{\calO_1^\times}{\{\pm 1\}} \hookrightarrow \frac{N_B(\calO)}{F^\times} \]
lies in the subgroup $N_{A}(\Lambda^{\times})/E^{\times}$ via
\begin{equation} \label{deltaembed}
\begin{aligned}
\Deltabar &\hookrightarrow \frac{N_{A}(\Lambda)}{E^{\times}} \hookrightarrow \frac{N_B(\calO)}{F^\times} \\
\deltabar_s &\mapsto \delta_s^2+1, \quad & \text{if $s \neq 2$}; \\
\deltabar_a &\mapsto (\delta_b^2+1)(\delta_c^2+1), \quad & \text{ if $a=2$}.
\end{aligned}
\end{equation}
where $s=a,b,c$ and $N$ denotes the normalizer.  The map \textup{(\ref{deltaembed})} extends the natural embedding \textup{(\ref{delta2barembed})}.
\end{prop}

\begin{proof}[Proof of Proposition \textup{\ref{defnembedyeah}}]
First, suppose $a \neq 2$ (whence $b,c \neq 2$, by the assumption that $a \leq b \leq c$).  In $B$, for each $s=a,b,c$, we have
\begin{equation} \label{wow}
\delta_s^2+1=\lambda_{2s}\delta_s;
\end{equation}
since $s \neq 2$, so that $\lambda_{2s} \neq 0$, this implies that $\delta_s^2+1$ has order $s$ in $A^{\times}/E^{\times} \subseteq B^\times/F^\times$ and
\[ (\delta_a^2+1)(\delta_b^2+1)(\delta_c^2+1)=\lambda_{2a}\lambda_{2b}\lambda_{2c}\delta_a\delta_b\delta_c=-\lambda_{2a}\lambda_{2b}\lambda_{2c} \in E^\times, \]
so (\ref{deltaembed}) defines a group homomorphism $\Deltabar \hookrightarrow A^{\times}/E^{\times}$.  The image lies in the normalizer $N_{A}(\Lambda)$ because $\Delta^{(2)}$ generates $\Lambda$ and $\Delta$ normalizes $\Delta^{(2)}$.  Finally, we have
\[ (\delta_s^2+1)^2=\lambda_{2s}^2\delta_s^2 \in A, \]
so the map extends the natural embedding of $\Delta^{(2)}/\{\pm 1\}$.

If $a=2$, the same argument applies, with instead
\[ \deltabar_a \mapsto (\delta_b^2+1)(\delta_c^2+1) \]
since $(\delta_b^2+1)(\delta_c^2+1)=\lambda_{2b}\lambda_{2c}(-\delta_a^{-1})=\lambda_{2b}\lambda_{2c}\delta_a$ now has order $2$ in $A^{\times}/E^{^\times}$, and necessarily $b,c>2$ since the triple is hyperbolic.
\end{proof}

\begin{exm}
The triangle group $\Delta(2,4,6)$ has trace field $F=\Q(\sqrt{2},\sqrt{3})$.  However, the group $\Delta(2,4,6)^{(2)}$ has trace field $E=\Q$ and indeed we find an embedding $\Deltabar(2,4,6) \hookrightarrow N_{A}(\Lambda)/\Q^\times$ where $\Lambda$ is a maximal order in a quaternion algebra $A$ of discriminant $6$ over $\Q$.
\end{exm}

\begin{cor} \label{pnmidO}
The following statements hold.
\begin{enumalph}
\item We have $\Lambda \otimes_{\Z_{E}} \Z_F \subseteq \calO$.
\item If $a \neq 2$, the quotient $\calO/(\lambda \otimes_{\Z_E} \Z_F)$ is annihilated by $\lambda_{2a} \lambda_{2b} \lambda_{2c}$.
\item If $a =2$, the quotient $\calO/(\lambda \otimes_{\Z_E} \Z_F)$ is annihilated by $\lambda_{2b} \lambda_{2c}$.
\end{enumalph}
\end{cor}

\begin{proof}
This follows from (\ref{wow}) since a basis for $\calO$ is given by $1,\delta_{a},\delta_{b},\delta_{c}$.
\end{proof}

We now define congruence subgroups of triangle groups.  Let $\frakN$ be an ideal of $\Z_F$ such that $\frakN$ is coprime to $abc$ and either $\frakN$ is coprime to $2$ or
\begin{center}
$(a,b,c) \neq (mk,m(k+1),mk(k+1))$ with $m,k \in \Z$.
\end{center}
Then by Lemma \ref{pnmid}, we have an isomorphism
\begin{equation} \label{thesplitting}
\calO \otimes_{\Z_F} \Z_{F,\frakN} \cong \M_2(\Z_{F,\frakN})
\end{equation}
where $\Z_{F,\frakN}$ denotes the $\frakN$-adic completion of the ring $\Z_F$: this is the product of the completions at $\frakP$ for $\frakP \mid \frakN$ and thus is a finite product of discrete valuation rings.  Any two maximal orders in a split quaternion algebra over a discrete valuation ring $R$ with fraction field $K$ are conjugate by an element of $M_2(K)$ \cite[Th\'eor\`eme II.2.3]{Vigneras}, and it follows easily
that the isomorphism (\ref{thesplitting}) is unique up to conjugation by an element of $\GL_2(\Z_{F,\frakN})$.  Alternately (and perhaps more fundamentally) since the ring $\Z_{F,\frakN}$ has trivial Picard group, the result follows from a generalization of the Noether-Skolem Theorem \cite[Corollary 12]{Rosenberg-Zelinsky}.

Let
\begin{equation} \label{eqn:OfrakN}
\calO(\frakN)=\{\gamma \in \calO : \gamma \equiv 1 \psmod{\frakN \calO}\}.
\end{equation}
The definition of $\calO(\frakN)$ does not depend on the choice of isomorphism in (\ref{thesplitting}).  Then $\calO(\frakN)_1^\times$ is normal in $\calO_1^\times$ and we have an exact sequence
\[ 1 \to \calO(\frakN)_1^\times \to \calO_1^\times/\{\pm 1\} \to \PSL_2(\Z_F/\frakN) \to 1 \]
where surjectivity follows from strong approximation \cite[Th\'eor\`eme III.4.3]{Vigneras}.  Let
\[ \Deltabar(\frakN) = \Deltabar \cap \calO(\frakN)_1^\times.  \]
Then we have
\begin{equation} \label{maptofrakP}
\frac{\Deltabar}{\Deltabar(\frakN)} \hookrightarrow \frac{\calO_1^{\times}/\{\pm 1\}}{\calO(\frakN)_1^{\times}} \cong \PSL_2(\Z_F/\frakN).
\end{equation}
We conclude by considering the image of the embedding (\ref{maptofrakP}).  Let $\frakn$ be the prime of $E=F(a,b,c)$ below $\frakN$.  Then $\frakn$ is coprime to the discriminant of $\Lambda$ since the latter divides $(\lambda_{2a}\lambda_{2b}\lambda_{2c})\beta$ by Corollary \ref{pnmidO}.  Therefore, we may define $\Lambda(\frakn)$ analogously.  Then by Proposition \ref{defnembedyeah}, we have an embedding
\begin{equation} \label{PGLwoah}
\Deltabar \hookrightarrow \frac{N_{A}(\Lambda)}{E^{\times}} \hookrightarrow \frac{A^{\times}}{E^{\times}} \hookrightarrow \frac{A_\frakn^{\times}}{E_\frakn^{\times}} \cong \PGL_2(E_\frakn)
\end{equation}
where $E_\frakn$ denotes the completion of $E$ at $\frakn$.  The image of $\Deltabar$ in this map lies in $\PGL_2(\Z_{E,\frakn})$ by \eqref{wow} since $\lambda_{2s} \in \Z_{E,\frakn}^\times$ for $s=a,b,c$ (since $\frakn$ is coprime to $abc$).  Reducing the image in (\ref{PGLwoah}) modulo $\frakn$, we obtain a map
\[ \Deltabar \to \PGL_2(\Z_E/\frakn). \]
This map is compatible with the map $\Deltabar \to \PSL_2(\Z_F/\frakN)$ inside $\PGL_2(\Z_F/\frakN)$, obtained by comparing the images in the reduction modulo $\frakN$ of $B^\times/F^\times$, by Proposition \ref{defnembedyeah}.

We summarize the main result of this section.

\begin{prop} \label{pglyeah}
Let $a,b,c \in \Z_{\geq 2} \cup \{\infty\}$.  Let $\frakN$ be an ideal of $\Z_F$ with $\frakN$ prime to $abc$ and such that either $\frakN$ is prime to $2$ or $(a,b,c) \neq (mk,m(k+1),mk(k+1))$ with $m,k \in \Z$.  Let $\frakn=\Z_E \cap \frakN$.  Then there exists a homomorphism
\[ \phi: \Deltabar(a,b,c) \to \PSL_2(\Z_F/\frakN) \]
such that $\tr \phi(\deltabar_s) \equiv \pm \lambda_{2s} \pmod{\frakN}$ for $s=a,b,c$.  The image of $\phi$ lies in the subgroup
\[ \PGL_2(\Z_E/\frakn) \cap \PSL_2(\Z_F/\frakN) \subseteq \PGL_2(\Z_F/\frakN). \]
\end{prop}

\begin{rmk} \label{TheoremBext}
We conclude this section with some remarks extending the primes $\frakP$ of $F$ (equivalently, primes $\frakp$ of $E$) for which the construction applies.

First, we note that whenever $\frakP \nmid \beta$, the order $\calO$ is maximal at $\frakP$.

Second, even for a ramified prime $\frakP$ (or $\frakp$), we still can consider the natural map to the completion; however, instead of $\PGL_2(F_\frakP)$ we instead obtain the units of an order in a division algebra over $F_\frakP$, a prosolvable group.  Our interest remains in the groups $\PSL_2$ and $\PGL_2$, but this case also bears further investigation: see Takei \cite{Takei} for some results in the case where $b=c=\infty$.


Third, we claim that
\[ B \cong \quat{\lambda_{2s}^2-4,\beta}{F} \]
for any $s \in \{a,b,c\}$.  Indeed, given the basis $1,\delta_a,\delta_b,\delta_c$, we construct an orthogonal basis for $B$ as
\[ 1, 2\delta_a-\lambda_{2a}, (\lambda_{2a}^2 - 4)\delta_b + (\lambda_{2a}\lambda_{2b}+2\lambda_{2c})\delta_a - (\lambda_{2a}^2\lambda_{2b} - \lambda_{2a}\lambda_{2c} + 2\lambda_{2b}) \]
which gives rise to the presentation $B \cong \quat{4-\lambda_{2a}^2,\beta}{F}$.  The others follow by symmetry.  It follows that a prime $\frakP$ of $\Z_F$ ramifies in $B$ if and only if we have for the Hilbert symbol (quadratic norm residue symbol) $(\lambda_{2s}^2-4,\beta)_{\frakP}=-1$ for (any) $s \in \{a,b,c\}$.  For example, if $(a,b,c)=(2,3,c)$ (with $c \geq 7$), one can show that the quaternion algebra $B$ is ramified at no finite place.


A similar argument \cite[Proposition 2]{Takeuchi2} shows that
\[ A \cong \quat{\lambda_{2b}^2(\lambda_{2b}^2-4), \lambda_{2b}^2\lambda_{2c}^2 \beta}{E}. \]
For any prime $\frakp$ of $E$ which is unramified in $A$, we can repeat the above construction, and we obtain a homomorphism $\phi$ as in (\ref{pglyeah}); the image can be analyzed by considering the isomorphism class of the local order $\Lambda_\frakp$, measured in part by the divisibility of $\beta$ by $\frakp$.
\end{rmk}

\section{Weak rigidity}

In this section, we investigate some weak forms of rigidity and rationality for Galois covers of $\PP^1$.  We refer to work of Coombes and Harbater \cite{CoombesHarbater}, Malle and Matzat \cite{MalleMatzat}, Serre \cite[Chapters 7--8]{TGT}, and Volklein \cite{Volklein} for references.  Our main result concerns three-point covers, but we begin by briefly considering more general covers.

Let $G$ be a finite group.  An $n$-\defi{tuple} for $G$ is a finite sequence $\gunder = (g_1,\ldots,g_n)$ of elements of
$G$ such that $g_1 \cdots g_n = 1$.  In our applications we will take $n = 3$, so we will not emphasize the dependence on $n$, and refer to \emph{tuples}.  A tuple is \defi{generating} if
$\langle g_1,\ldots,g_n \rangle = G$.  Let $\Cunder = (C_1,\ldots,C_n)$ be a finite sequence of conjugacy classes of $G$.  Let $\Sigma(\Cunder)$ be the set of generating tuples $\gunder = (g_1,\ldots,g_n)$ such that $g_i \in C_i$ for all $i$.

The group $\Inn(G) = G/Z(G)$ of inner automorphims of $G$ acts on $G^n$ via
\[ x \cdot \gunder = x \cdot (g_1,\ldots,g_n) = \gunder^x = (xg_1 x^{-1},\ldots,x g_n x^{-1}) \]
and restricts to an action of $\Inn(G)$ on $\Cunder$.

Suppose that $G$ has trivial center, so $\Inn(G) = G$.  To avoid trivialities, suppose also that $\Sigma(\Cunder) \neq \emptyset$.  Then the action of $\Inn(G)$ on $\Sigma(\Cunder)$ has no fixed points: if $z \in G$ fixes $\gunder$, then $z$ commutes with each $g_i$ hence with $\langle g_1,\ldots,g_n \rangle = G$, so $z \in Z(G) = \{1\}$.

Now suppose that $n=3$; we call a $3$-tuple a \defi{triple}.  For every generating triple $\gunder$, we obtain from the Riemann Existence Theorem \cite[Thm. 2.13]{Volklein} a $G$-Galois branched covering $X(\gunder) \to \PP^1$ defined over $\Qbar$ with ramification type $\gunder$ over $0,1,\infty$ and Galois group $G$.  Two such covers $f:X(\gunder) \to \PP^1$ and $f':X(\gunder') \to \PP^1$ are \defi{isomorphic} as covers if there exists an isomorphism $h:X(\gunder) \xrightarrow{\sim} X(\gunder')$ such that $f=f'\circ h$; such an isomorphism from $f$ to $f'$ corresponds to an element $\varphi \in \Aut(G)$ such that $\varphi(\gunder)=(\varphi(g_1),\dots,\varphi(g_n))= \gunder'$, and conversely.

We will have need also of a more rigid notion.  A \defi{$G$-Galois branched cover} is a branched cover $f:X \to \PP^1$ equipped with an isomorphism $i:G \xrightarrow{\sim} \Aut(X,f)$.  Two $G$-Galois branched covers $(f,i)$ and $(f',i')$ are \defi{isomorphic} (as $G$-Galois branched covers) if and only if there is an isomorphism from $f$ to $f'$ that maps $i$ to $i'$; such an isomorphism corresponds to an element $x \in G$ such that
\[ \gunder^x = (\gunder')^x \]
and conversely.

The group $\Gal(\Qbar/\Q)$ acts on the set of generating tuples for $G$ up to automorphism (or simply inner automorphism) via its action on the covers.  Coming to grips with the mysteries of this action in general is part of Grothendieck's program of \emph{dessin d'enfants} \cite{Grothendieck}: to understand $\Gal(\Qbar/\Q)$ via its faithful action on the fundamental group of $\PP^1_{\overline{\Q}} \setminus \{0,1,\infty\}$.  There is one part of the action which is understood, coming from the maximal abelian extension of $\Q$ generated by roots of unity.

Let $\zeta_s=\exp(2\pi i/s) \in \C$ be a primitive $s$th root of unity for $s \in \Z_{\geq 2}$.  The group $\Gal(\Q^{\textup{ab}}/\Q)$ acts on tuples via the cyclotomic character $\chi$: for $\sigma \in \Gal(\Q^{\textup{ab}}/\Q)$ and a triple $\gunder$, we have $\sigma \cdot \gunder$ is uniformly conjugate to $(g_1^{\chi(\sigma)},\dots,g_n^{\chi(\sigma)})$ where if $g_i$ has order $m_i$ then  $g_i^{\chi(\sigma)}$ is conjugate to $g_i^{a_i}$, where $\sigma(\zeta_{m_i})=\zeta_{m_i}^{a_i}$.  This action becomes an action on conjugacy classes in purely group theoretic language as follows.  Let $m$ be the exponent of $G$.  Then the group $(\Z/m\Z)^{\times}$ acts on $G$ by $s \cdot g = g^s$ for $s \in (\Z/m\Z)^\times$ and $g \in G$ and this induces an action on conjugacy classes.  Pulling back by
the canonical isomorphism $\Gal(\Q(\zeta_m)/\Q) \xrightarrow{\sim} (\Z/m \Z)^{\times}$ defines the action of
$\Gal(\Q(\zeta_m)/\Q)$ and hence also $\Gal(\Q^{\textup{ab}}/\Q)$ on the set of triples for $G$.

Let $H\sbrat \subseteq \Gal(\Q(\zeta_m)/\Q)$ be the kernel of this action:
\[ H\sbrat = \{ s \in (\Z/m\Z)^\times : C^s=C \text{ for all conjugacy classes $C$}\}. \]
The fixed field $F\sbrat(G) = \Q(\zeta_m)^{H\sbrat}$ is called the \defi{field of rationality} of $G$.  The field $F\sbrat(G)$ can also be
characterized as the field obtained by adjoining to $\Q$ the values of the character table of $G$.
Let
\[ H(\Cunder)=\{s \in (\Z/m\Z)^{\times} : \text{ $C_i^{s} = C_i$ for all $i$}\} \]
be the stabilizer of $\Cunder$ under this action.  We define the \defi{field of rationality} of $\Cunder$ to be
\[ F\sbrat(\Cunder) =  \Q(\zeta_m)^{H(\Cunder)}. \]
Similarly, let
\[ H\sbwk(\Cunder) = \{s \in (\Z/m\Z)^{\times} : \Cunder^s=\varphi(\Cunder) \text{ for some $\varphi \in \Aut(G)$}\}. \]
We define the \defi{field of weak rationality} of $\Cunder$ to be $F\sbwk(\Cunder) = \Q(\zeta_m)^{H\sbwk(\Cunder)}$.  Then \[F\sbwk(\Cunder) \subseteq F\sbrat(\Cunder) \subseteq F\sbrat(G). \]

The group $\Gal(\Qbar/F\sbwk(\Cunder))$ acts on the set of generating tuples $\gunder \in \Sigma(\Cunder)$ up to uniform automorphism, which we denote $\Sigma(\Cunder)/\!\Aut(G)$.  For $\gunder \in \Sigma(\Cunder)$, the cover $f:X=X(\gunder) \to \PP^1$ has field of moduli $M(X,f)$ equal to the fixed field of the kernel of this action, a number field of degree at most  $d\sbwk=\# \Sigma(\Cunder)/\!\Aut(G)$ over $F\sbwk(\Cunder)$.

Similarly, a $G$-Galois branched cover $f:X \to \PP^1$ (equipped with its isomorphism $i:G \xrightarrow{\sim} \Aut(X,f)$) has field of moduli $M(X,f,G)$ equal to the fixed field of the stabilizer of the action of $\Gal(\Qbar/F\sbrat(\Cunder))$, a number field of degree $\leq d\sbrat = \#\Sigma(\Cunder)/\Inn(G)$ over $F\sbrat(\Cunder)$.

Therefore, we have the following diagram of fields.
\[
\xymatrix{
  &        & M(X,f,G) \ar@{-}[dd]^{\leq d\sbrat=\#\Sigma(\Cunder)/\Inn(G)} \ar@{-}[dl] \\
&M(X,f) \ar@{-}[dd]_{\leq d\sbwk=\#\Sigma(\Cunder)/\!\Aut(G)} & \\
&& F\sbrat(\Cunder) \ar@{-}[dl] \\
&F\sbwk(\Cunder) \ar@{-}[d] \\
&\Q
} \]

The simplest case of this setup is as follows.  We say that $\Cunder$ is \defi{rigid} if the action of $\Inn(G)$ on $\Sigma(\Cunder)$ is transitive.  By the above, if $\Sigma(\Cunder)$ is rigid then this action is simply transitive and so endows $\Sigma(\Cunder)$ with the structure of a torsor under $G = \Inn(G)$.  In this case, the diagram collapses to
\[ M(X,f,G)=F\sbrat(\Cunder) \supseteq M(X,f)=F\sbwk(\Cunder). \]
More generally, we say that $\Cunder$ is \defi{weakly rigid} if for all $\gunder, \gunder' \in
\Sigma(\Cunder)$ there exists $\varphi \in \Aut(G)$ such that $\varphi(\gunder) = \gunder'$.  (Coombes and Harbater \cite{CoombesHarbater} say \defi{inner rigid} and \defi{outer rigid} for rigid and weakly rigid, respectively.)  If $\underline{C}$ is weakly rigid, and $X=X(\underline{g})$ with $\underline{g} \in \underline{C}$, then $M(X,f)=F\sbwk(\Cunder)$ and the group $\Gal(M(X,f,G)/F\sbrat)$ injects canonically into the outer automorphism group $\Out(G) = \Aut(G)/\Inn(G)$.

We summarize the above discussion in the following proposition.

\begin{prop} \label{WRWR}
Let $G$ be a group with trivial center.  Let $\gunder=(g_1,\dots,g_n)$ be a generating tuple for $G$ and let $\Cunder=(C_1,\dots,C_n)$, where $C_i$ is the conjugacy class of $g_i$.  Let $P_1,\dots,P_n \in \PP^1(\overline{\Q})$.  Then the following statements hold.
\begin{enumalph}
\item There exists a branched covering $f:X \to \PP^1$ with ramification type $\Cunder=(C_1,\dots,C_n)$ over the points $P_1,\dots,P_n$
and an isomorphism $G \xrightarrow{\sim} \Aut(X,f)$, all defined over $\Qbar$.
\item The field of moduli $M(X,f)$ of $f$ is a number field of degree at most
\[ d\sbwk=\#\Sigma(\Cunder)/\!\Aut(G) \]
over $F\sbwk(\Cunder)$.
\item The field of moduli $M(X,f,G)$ of $f$ as a $G$-Galois branched cover is a number field of degree at most
\[ d\sbrat=\#\Sigma(\Cunder)/\Inn(G) \]
over $F\sbrat(\Cunder)$.
\end{enumalph}
\end{prop}

\section{Conjugacy classes, fields of rationality} \label{sec:conjclass}

Let $p$ be a prime number and $q = p^r$ a prime power.  Let $\F_q$ be a field with $q$ elements and algebraic closure
$\overline{\F}_q$.  In this section, we record some basic but crucial
facts concerning conjugacy classes and automorphisms in the finite matrix groups arising from $\GL_2(\F_q)$; see Huppert \cite[\S II.8]{Huppert} for a reference.

First let $g \in \GL_2(\F_q)$.  By the Jordan canonical form, exactly one of the following holds:
\begin{enumerate}
\item The characteristic polynomial $f(g;T) \in \F_q[T]$ has two repeated roots (in $\F_q$), and hence $g$ is either a scalar matrix (central in $\GL_2(\F_q)$) or $g$ is conjugate to a matrix of the form $\begin{pmatrix} t & 1 \\ 0 & t \end{pmatrix}$ with $t \in \F_q^\times$; or
\item $f(g;T)$ has distinct roots (in $\overline{\F}_q$) and the conjugacy class of $g$ is uniquely determined by $f(g;T)$, and we say $g$ is \defi{semisimple}.
\end{enumerate}

Let $\PGL_2(\F_q)=\GL_2(\F_q)/\F_q^\times$ and let $\overline{g}$ be the image of $g$ under the natural reduction map $\GL_2(\F_q) \rightarrow \PGL_2(\F_q)$.  We have $\overline{g}=\overline{1}$ iff $g$ is a scalar matrix.  If $g$ is conjugate to $\begin{pmatrix} t & 1 \\ 0 & t \end{pmatrix}$, then $\overline{g}$ is conjugate to $\begin{pmatrix} 1 & 1 \\ 0 & 1 \end{pmatrix}$, and we say that $\overline{g}$ is \defi{unipotent}.  If $f(g;T)$ is semisimple, then in the quotient the conjugacy classes associated to $f(g;T)$ and $f(cg;T)=c^2f(g;c^{-1}T)$ for $c \in \F_q^\times$ become identified.  If $f(g;T)$ factors over $\F_q$ then $\overline{g}$ is conjugate in $\PGL_2(\F_q)$ to the image of a matrix $\begin{pmatrix} 1 & 0 \\ 0 & x \end{pmatrix}$ with $x \in \F_q^\times \setminus \{1\}$, and we say that $\overline{g}$ is \defi{split (semisimple)}.  The set of split semisimple conjugacy classes in $\PGL_2(\F_q)$ is therefore in bijection with the set
\begin{equation} \label{eqn:conjxxinv}
\{\{x,x^{-1}\} : x \in \F_q^\times \setminus \{1\}\}.
\end{equation}
There are $(q-3)/2+1=(q-1)/2$ such classes if $q$ is odd, and $(q-2)/2=q/2-1$ such classes if $q$ is even.  On the other hand, the conjugacy classes of semisimple elements with irreducible $f(g;T)$ are in bijection with the set of monic, irreducible polynomials $f(T) \in \F_q[T]$ of degree $2$ up to rescaling $x \mapsto ax$ with $a \in \F_q^\times$.  There are $q(q-1)/2$ such monic irreducible quadratic polynomials $T^2-aT+b$; for any such polynomial with $a \neq 0$, there is a unique rescaling such that $a=1$; when $q$ is odd, there is a unique such polynomial with $a=0$ up to rescaling.  Therefore, the total number of conjugacy classes is $(q(q-1)/2)/(q-1)=q/2$ when $q$ is even and $\left(q(q-1)/2-(q-1)/2\right)/(q-1)+1=(q-1)/2$ when $q$ is odd.  Equivalently, the set of nonsplit semisimple conjugacy classes in $\PGL_2(\F_q)$ is in bijection with the set
\begin{equation} \label{eqn:conjyyq}
\{\{y,y^q\} : y \in (\F_{q^2} \setminus \F_q)/\F_q^\times\}
\end{equation}
by taking roots.

Now let $g \in \SL_2(\F_q) \setminus $ with $g \neq \pm 1$.  Suppose first that $f(g;T)$ has a repeated root, necessarily $\pm 1$; then we say that $g$ is \defi{unipotent}.  For $u \in \F_q$, let $U(u) =  \begin{pmatrix} 1 & u \\ 0 & 1 \end{pmatrix}$.  Using Jordan canonical form we find that $g$ is conjugate to $\pm U(u)$ for some $u \in \F_q^{\times}$.  The matrices $U(u)$ and $U(v)$ are conjugate if and only if $uv^{-1} \in \F_q^{\times 2}$.
Thus, if $q$ is odd there are four nontrivial conjugacy classes associated to characteristic polynomials with repeated roots, whereas is $q$ is even there is a single such conjugacy class.

Otherwise the element $g$ is semisimple and so $g$ is conjugate in $\SL_2(\F_q)$ to the matrix $\begin{pmatrix} 0 & -1 \\ 1 & \tr(g) \end{pmatrix}$ by rational canonical form, and the trace map provides a bijection between the set of conjugacy classes of semisimple elements of $\SL_2(\F_q)$ and elements $\alpha \in \F_q$ with $\alpha \neq \pm 2$.

Finally, we give the corresponding description in $\PSL_2(\F_q) = \SL_2(\F_q)/\{\pm 1\}$.  When $p=2$ we have $\PSL_2(\F_q) = \SL_2(\F_q)$, so assume that $p$ is odd.  Then the conjugacy classes of the matrices $U(u)$ and $-U(u)$ in $\SL_2(\F_q)$ become identified in $\PSL_2(\F_q)$, so there are precisely two nontrivial \defi{unipotent} conjugacy classes, each consisting of elements of order $p$.  If $g$ is a semisimple element of $\SL_2(\F_q)$ of order $a$, then the order of its image $\pm g$ in $\PSL_2(\F_q)$ is $a/\!\gcd(a,2)$.  We define the \defi{trace} of an element $\pm g \in \PSL_2(\F_q)$ to be $\tr(\pm g)=\{\tr(g),-\tr(g) \} \subseteq \F_q$ and define the \defi{trace field} of $\pm g$ to be $\F_p(\tr(\pm g))$.  The conjugacy class, and therefore the order, of a semisimple element of $\PSL_2(\F_q)$ is then again uniquely determined by its trace.  (This is particular to $\PSL_2(\F_q)$---the trace does not determine a conjugacy class in $\PGL_2(\F_q)$!)

We now describe outer automorphism groups (see e.g.\ Suzuki \cite{Suzuki}).  The $p$-power Frobenius map $\sigma$, acting on the entries of a matrix by $a \mapsto a^p$, gives an outer automorphism of $\PSL_2(\F_q)$ and $\PGL_2(\F_q)$, and in fact $\Out(\PGL_2(\F_q)) = \langle \sigma \rangle$.  When $p$ is odd, the map $\tau$ given by conjugation by an element in $\PGL_2(\F_q) \setminus \PSL_2(\F_q)$ is also an outer automorphism of $\PSL_2(\F_q)$, and these maps generate $\Out(\PSL_2(\F_q))$:
\begin{equation} \label{outpsl2fq}
\Out(\PSL_2(\F_q)) \cong
\begin{cases}
\la \sigma,\tau \ra, &\text{ if $p$ is odd;} \\
\la \sigma \ra, &\text{ if $p=2$}.
\end{cases}
\end{equation}
In particular, the order of $\Out(\PSL_2(\F_q))$ is $2r$ if $p$ is odd and $r$ if $p=2$.  From the  embedding $\PGL_2(\F_q) \hookrightarrow \PSL_2(\F_{q^2})$, given explicitly by $\pm g \mapsto \pm(\det g)^{-1/2} g$, we may also view the outer automorphism $\tau$ as conjugation by an element of $\PSL_2(\F_{q^2}) \setminus \PSL_2(\F_q)$.

We conclude this section by describing the field of rationality (as defined in \S 5) for these conjugacy classes.

\begin{lem} \label{PGL2fieldofrat}
Let $\overline{g} \in \PGL_2(\F_q)$ have order $m$.  Then the field of rationality of the conjugacy class $C$ of $\overline{g}$ is
\[
F\sbrat(C)=
\begin{cases}
\Q(\lambda_m), & \text{if $\overline{g}$ is semisimple}; \\
\Q, & \text{if $\overline{g}$ is unipotent};
\end{cases} \]
and the field of weak rationality of $C$ is
\[
F\sbwk(C)=
\begin{cases}
\Q(\lambda_m)^{\la \Frob_p \ra}, & \text{if $\overline{g}$ is semisimple}; \\
\Q, & \text{if $\overline{g}$ is unipotent}.
\end{cases} \]
\end{lem}

\begin{proof}
A power of a unipotent conjugacy class is unipotent or trivial so its field of rationality and weak rationality is $\Q$.

If $C$ is split semisimple, corresponding to $\{x,x^{-1}\}$ by \eqref{eqn:conjxxinv} with $x \in \F_q^\times \setminus \{1\}$ then $\overline{g} \mapsto \overline{g}^s$ for $s \in (\Z/m\Z)^\times$ corresponds to the map $x \mapsto x^s$ and it stabilizes the set $\{x,x^{-1}\}$ (resp.\ up to an automorphism of $\PGL_2(\F_q)$) if and only if $s \in \la -1 \ra \subseteq (\Z/m\Z)^\times$ (resp.\  $s \in \la -1,p \ra$).

Next consider the case where $C$ is nonsplit semisimple, corresponding to $\{y,y^q\}$ by \eqref{eqn:conjyyq} with $y \in (\F_{q^2} \setminus \F_q)/\F_q^\times$.  Then the map $\overline{g} \mapsto g^s$ again with $s \in (\Z/m\Z)^\times$ corresponds to the map $y \mapsto y^s$.  We have $y \in \F_{q^2}^\times \cong \Z/(q^2-1)\Z$, with the image of $\F_q^\times$ the subgroup $(q+1)\Z/(q^2-1)\Z$, so the set $\{y,y^q\}$ is stable if and only if $s \in \{1,q\} = \langle -1 \rangle \pmod{m}$, and the set is stable up to an automorphism of $\PGL_2(\F_q)$ if and only if $s \in \la -1,p \ra \subset \Z/m\Z$, so we have the same result as in the split case.
\end{proof}

For an odd prime $p$, we abbreviate $p^*=(-1)^{(p-1)/2} p$.  Recall that $q=p^r$.

\begin{lem} \label{PSL2fieldofrat}
Let $\pm g \in \PSL_2(\F_q)$ have order $m$.  Then the field of rationality of the conjugacy class $C$ of $g$ is
\[
F\sbrat(C)=
\begin{cases}
\Q(\lambda_m), & \text{if $g$ is semisimple}; \\
\Q(\sqrt{p^*}), & \text{if $g$ is unipotent and $pr$ is odd}; \\
\Q, & \text{otherwise}.
\end{cases} \]
The field of weak rationality of $C$ is
\[
F\sbwk(C)=
\begin{cases}
\Q(\lambda_m)^{\la \Frob_p \ra}, & \text{if $g$ is semisimple}; \\
\Q, & \text{otherwise},
\end{cases} \]
where $\Frob_p \in \Gal(\Q(\lambda_m)/\Q)$ is the Frobenius element associated to the prime $p$.
\end{lem}

\begin{proof}
First, suppose $\pm g = \pm U(u)$ is unipotent with $u \in \F_q^{\times}$.  Then for all
integers $s$ prime to $p$, we have $(\pm g)^s=\pm U(su)$.  Thus, the subgroup of $(\Z/p\Z)^{\times} = \F_p^{\times}$ stabilizing $C$ is precisely the set of elements of $\F_p^{\times}$ which are squares in $\F_q^\times$.  Thus if $p = 2$ or $r$ is even, this subgroup is all of $\F_p^{\times}$ so that the field of rationality of $C$ is $\Q$, whereas if $pr$ is odd this subgroup is the unique index two subgroup of $\F_p^{\times}$ and the corresponding field of rationality for $C$ in $\PSL_2(\F_q)$ is $\Q(\sqrt{p^*})$.

Next we consider semisimple conjugacy classes.  By the trace map, these classes are in bijection with $\pm t \in \pm\F_q \setminus \{\pm 2\}$.  The induced action on the set of traces is given by $\pm t = \pm(z+1/z) \mapsto \pm(z^s+1/z^s)$ for $s \in (\Z/m\Z)^\times$ where $z$ is a primitive $m$th root of unity.  From this description, we see that the stabilizer is $\la -1 \ra \subseteq (\Z/m\Z)^\times$.

A similar analysis yields the field of weak rationality.  If $C$ is unipotent then $\tau$ identifies the two unipotent conjugacy classes so the field of weak rationality is always $\Q$.  If $C$ is semisimple then $\sigma$ identifies $C$ with $C^p$ so the stabilizer of $ \pm t$ is $\la -1, p \ra \subseteq (\Z/m\Z)^\times$, the field fixed further under the Frobenius $\Frob_p$.
\end{proof}

\section{Subgroups of $\PSL_2(\F_q)$ and $\PGL_2(\F_q)$ and weak rigidity}

The general theory developed for triples in $\S 5$ can be further applied to the groups $\PSL_2(\F_q)$ (and consequently $\PGL_2(\F_q)$) using work of Macbeath \cite{Macbeath}, which we recall in this section.  See also Langer and Rosenberger \cite{LangerRosenberger}, who give an exposition of Macbeath's work in our context.

Let $q$ be a prime power.  We begin by considering triples $\gunder=(g_1,g_2,g_3)$ with $g_i \in \SL_2(\F_q)$ -- we remind the reader that
the terminology implies $g_1 g_2 g_3 = 1$ and does not imply that $\langle g_1,g_2,g_3 \rangle = \SL_2(\F_q)$ --
 with an eye to understanding the image of the subgroup generated by $g_1,g_2,g_3$ in $\PSL_2(\F_q)$ according to the traces of the corresponding elements.  Long periods of consternation have taught us that the difference between a matrix and a matrix up to sign plays an important role here, and so we keep this in our notation.  Moreover, because we will be considering other kinds of triples, we refer to $\gunder=(g_1,g_2,g_3)$ as a \defi{group triple}.

A \defi{trace triple} is a triple $\tunder=(t_1,t_2,t_3) \in \F_q^3$.  For a trace triple $\tunder$, let $T(\tunder)$ denote the set of group triples $\gunder$ such that $\tr(g_i)=t_i$ for $i=1,2,3$.  The group $\Inn(\SL_2(\F_q)) = \PSL_2(\F_q)$ acts on $T(\tunder)$ by conjugation.

\begin{prop}[{Macbeath \cite[Theorem 1]{Macbeath}}] \label{MACTHM1}
For all trace triples $\tunder$, the set $T(\tunder)$ is nonempty.
\end{prop}

To a group triple $\gunder=(g_1,g_2,g_3) \in \SL_2(\F_q)^3$, we associate the \defi{order triple} $(a,b,c)$ by letting $a$ be the order of $\pm g_1 \in \PSL_2(\F_q)$, and similarly $b$ the order of $\pm g_2$ and $c$ the order of $\pm g_3$.  Without loss of generality, as in the definition of the triangle group (\ref{reorderremark}) we may assume that an order triple $(a,b,c)$ has $a \leq b \leq c$.

Our goal is to give conditions under which we can be assured that a group triple generates $\PSL_2(\F_q)$ or $\PGL_2(\F_q)$ and not a smaller group.  We do this by placing restrictions on the associated trace triples, which come in three kinds.

A trace triple $\tunder$ is \defi{commutative} if there exists $\gunder \in T(\tunder)$ such that the group $\pm \langle g_1, g_2, g_3 \rangle \subseteq \PSL_2(\F_q)$ is commutative.
By a direct calculation, Macbeath proves that a triple $\tunder$ is commutative if and only if the ternary $\F_q$-quadratic form
\[ x^2 + y^2 + z^2 + t_1 yz + t_2 xz + t_3 xy \]
is singular \cite[Corollary 1, p.\ 21]{Macbeath}, i.e.\ if and only if its (half-)discriminant
\begin{equation} \label{commform}
d(\tunder)=d(t_1,t_2,t_3)=t_1^2+t_2^2+t_3^2 - t_1t_2t_3 - 4
\end{equation}
is zero.  If a trace triple $\tunder$ is not commutative, then the order triple $(a,b,c)$ is the same for any $\gunder \in T(\tunder)$: the trace uniquely defines the order of a semisimple or unipotent element, and if some $g_i$ is scalar in $\gunder \in T(\tunder)$ then the group it generates is necessarily commutative.

A trace triple $\tunder$ is \defi{exceptional} if there exists a triple $\gunder \in T(\tunder)$ with order triple equal to $(2,2,c)$ with $c \geq 2$ or one of
\begin{equation} \label{exceptionaltriples}
(2,3,3), (3,3,3), (3,4,4), (2,3,4), (2,5,5), (5,5,5), (3,3,5), (3,5,5), (2,3,5).
\end{equation}
Put another way, a trace triple $\tunder$ is exceptional if there exists $\gunder \in T(\tunder)$ whose order triple is the same as that of a triple of elements of $\SL_2(\F_q)$ that generates a finite spherical triangle group in $\PSL_2(\F_q)$.

Finally, a trace triple $\tunder$ is \defi{projective} if for all $\gunder=(g_1,g_2,g_3) \in T(\tunder)$, the subgroup $\pm \la g_1, g_2, g_3 \ra \subseteq \PSL_2(\F_q)$ is conjugate to a subgroup of the form $\PSL_2(k)$ or $\PGL_2(k)$ for $k \subseteq \F_q$ a subfield.

\begin{rmk}
There are no trace triples which are both projective and commutative.  A projective trace triple may be exceptional, but the possibilities can be explicitly described as follows.  A trace triple is exceptional if there is a homomorphism from a finite spherical group to $\PSL_2(\F_q)$ with order triple as given in (\ref{exceptionaltriples}); a trace triple is projective if the image is conjugate to $\PSL_2(k)$ or $\PGL_2(k)$ for $k \subseteq \F_q$ a subfield.  From the classification of finite spherical groups, this homomorphism must be one of the following exceptional isomorphisms:
\begin{center}
$D_6 \cong \GL_2(\F_2)$, $A_4 \cong \PSL_2(\F_3)$, $S_4 \cong \PGL_2(\F_3)$, or $A_5 \cong \PGL_2(\F_4) \cong \PSL_2(\F_5)$.
\end{center}
Any trace triple $\underline{t}$ with $\F_p(\underline{t})=\F_p(t_1,t_2,t_3)=\F_q$ that is both exceptional and projective corresponds to one of these isomorphisms.
\end{rmk}

We now come to Macbeath's classification of subgroups of $\PSL_2(\F_q)$ generated by two elements.

\begin{thm}[{\cite[Theorem 4]{Macbeath}}] \label{MacbeathClass} \label{MACTHM4}
Every trace triple $\tunder$ is exceptional, commutative or partly projective.
\end{thm}

\begin{exm} \label{Ohyeahp7}
We illustrate the above with the case $q=7$.  There are a total of $7^3=243$ trace triples.

First, the trace triples where the order triple is \emph{not} well-defined are the trace triples
\[ (2,2,2),(2,-2,-2),(-2,2,-2),(-2,-2,2) \]
which are commutative.  The role of multiplication by $-1$ plays an obvious role here, so in this example, for a trace triple $\underline{t}=(t_1,t_2,t_3)$ we say that a trace triple \defi{agrees with $\underline{t}$ with an even number of signs} if it is one of
\[ (t_1,t_2,t_3),(t_1,-t_2,-t_3),(-t_1,t_2,-t_3),(-t_1,-t_2,t_3). \]
Put this way, the trace triples where the order triple is not well-defined are those agreeing with $(2,2,2)$ with an even number of signs.  We define \defi{odd number of signs} analogously.  For each of these four trace triples, there exists a group triple $\gunder$ with order triple $(1,1,1)$, $(1,7,7)$, or $(7,7,7)$.

The other commutative trace triples are:
\begin{align*}
(2, 0, 0) & \text{, with any signs} & \quad \text{having orders } (1,2,2) \\
(2,1,1) & \text{, with even number of signs} &\quad \text{having orders }(1,3,3) \\
(2,3,3) &\text{, with even number of signs} &\quad \text{having orders }(1,4,4) \\
(0,0,0) &\text{, with any signs} &\quad \text{having orders } (2,2,2) \\
(0,3,3) &\text{, with any signs} &\quad \text{having orders } (2,4,4) \\
(1,1,-1) &\text{, with odd number of signs} &\quad \text{having orders }(3,3,3)
\end{align*}
Indeed, these are all values $(t_1,t_2,t_3) \in \F_q^3$ such that
\[ d(t_1,t_2,t_3)=t_1^2+t_2^2+t_3^2-t_1t_2t_3-4=0. \]

The three commutative trace triples $(0,0,0),(0,3,3),(1,1,-1)$ are also exceptional.  The remaining exceptional triples are:
\begin{align*}
(0,0,1) &\text{, with any signs} & \quad \text{having orders } (2,2,3) \\
(0,0,3) &\text{, with any signs} & \quad \text{having orders } (2,2,4) \\
(0,1,1) &\text{, with any signs} & \quad \text{having orders } (2,3,3) \\
(0,1,3) &\text{, with any signs} & \quad \text{having orders }(2,3,4) \\
(1,1,1) &\text{, with even number of signs} &\quad \text{having orders }(3,3,3) \\
(1,3,3) &\text{, with any signs} &\quad \text{having orders }(3,4,4)
\end{align*}

All other triples are projective:
\begin{align*}
(0, 1, 2) & \text{, with any signs} & \quad \text{having orders } (2,3,7) \\
(0, 3, 2) & \text{, with any signs} & \quad \text{having orders } (2,4,7) \\
(0, 2, 2) & \text{, with any signs} & \quad \text{having orders } (2,7,7) \\
(1, 1, 3) & \text{, with any signs} & \quad \text{having orders } (3,3,4) \\
(1,1,2) & \text{, with any signs} & \quad \text{having orders } (3,3,7) \\
(1,3,2) & \text{, with any signs} & \quad \text{having orders }(3,4,7) \\
(1,2,2)& \text{, with any signs} & \quad \text{having orders } (3,7,7) \\
(3,3,3) & \text{, with any signs} & \quad \text{having orders } (4,4,4) \\
(3,3,-2) & \text{, with odd number of signs} & \quad \text{having orders } (4,4,7) \\
(3,2,2) & \text{, with any signs} & \quad \text{having orders } (4,7,7) \\
(2,2,-2) & \text{, with odd number of signs} & \quad \text{having orders } (7,7,7)
\end{align*}

We note that the triples $(1,3,-3)$ with an odd number of signs in fact generate $\PSL_2(\F_7)$---but the triple is not projective.  In particular, we observe that one cannot deduce that a nonsingular trace triple is projective by looking only at its order triple.

Finally, the issue that this example is supposed to make clear is that changing signs on a trace triple may change the subgroup of $\PSL_2(\F_q)$ that a corresponding group triple generates.  Indeed, changing an odd number of signs on a group triple does not yield a group triple!  We address the parity of signs in the next lemma.
\end{exm}

The role of $-1$ and the parity of these signs (taking an even or odd number) is a key issue that will arise and so we address it now.

\begin{lem} \label{commproj}
Let $\tunder=(t_1,t_2,t_3) \in \F_q^3$ be a trace triple.
\begin{enumalph}
\item There are bijections
\begin{align*}
T(\tunder) \leftrightarrow T(t_1,-t_2,-t_3) &\leftrightarrow T(-t_1,t_2,-t_3) \leftrightarrow T(-t_1,-t_2,t_3) \\
(g_1,g_2,g_3)  \mapsto (g_1,-g_2,-g_3) &\mapsto (-g_1,g_2,-g_3) \mapsto (-g_1,-g_2,g_3)
\end{align*}
which preserve the subgroups generated by each triple.  In particular, if $\tunder$ is commutative (resp.\ exceptional, projective), then so is each of
\[ (t_1,-t_2,-t_3),(-t_1,t_2,-t_3),(-t_1,-t_2,t_3). \]
\item Suppose $q$ is odd.  If $\tunder$ is commutative, then $(-t_1,t_2,t_3)$ is commutative if and only if $t_1t_2t_3=0$.
\end{enumalph}
\end{lem}

\begin{proof}
Part (a) is clear.  As for part (b): the trace triple $\tunder$ is commutative if and only if $d(t_1,t_2,t_3)=0$.  So $(-t_1,t_2,t_3)$ is also commutative if and only if $d(-t_1,t_2,t_3)=0$ if and only if $d(t_1,t_2,t_3)-d(-t_1,t_2,t_3)=2t_1t_2t_3=0$, as claimed.
\end{proof}


Let $\ell/k$ be a separable quadratic extension.  We say that $t \in \ell$ \defi{is a squareroot from $k$} if $t=0$ or $t=\sqrt{u}$ with $u \in k^\times \setminus k^{\times 2}$.  A trace triple $\tunder$ is \defi{irregular} \cite[p.\ 28]{Macbeath} if the field $\F_p(\tunder)=\F_p(t_1,t_2,t_3) \subseteq \F_q$ has a subfield $k \subseteq \F_p(\tunder)$ such that
\begin{enumroman}
\item $[\F_p(\tunder):k] = 2$ and
\item after reordering, we have $t_1 \in k$ and $t_2,t_3$ are squareroots from $k$.
\end{enumroman}
Otherwise, $\tunder$ is \defi{regular}.  Of course, if $[\F_p(\tunder):\F_p]$ is odd---e.g., if $[\F_q:\F_p]$ is odd---then $\tunder$ is necessarily regular.

\begin{prop}[{\cite[Theorem 3]{Macbeath}}] \label{MACTHM3}
Let $\gunder$ generate a projective subgroup $G=\pm \la g_1,g_2,g_3 \ra \subseteq \PSL_2(\F_q)$ and let $\tunder$ be its trace triple.
\begin{enumalph}
\item Suppose $\tunder$ is regular.  Then $G$ is conjugate
in $\PSL_2(\F_q)$ to $\PSL_2(\F_p(\tunder))$.
\item Suppose $\tunder$ is irregular, and let $k_0$ be the unique
index $2$ subfield of $\F_p(\tunder)$.  Then $G$ is conjugate in
$\PSL_2(\F_q)$ to either $\PSL_2(\F_p(\tunder))$ or $\PGL_2(k_0)$.
\item Suppose $k=\F_q$.  Then the number of orbits of $\Inn(\SL_2(\F_q)) = \PSL_2(\F_q)$ on $T(\tunder)$ is $2$ or $1$ according as $p$ is odd or $p=2$.
\item For all $\gunder' \in T(\tunder)$, there exists $m \in \SL_2(\overline{\F}_q)$
such that $m^{-1} \gunder m = \gunder'$.
\end{enumalph}
\end{prop}

We say that a trace triple $\tunder$ is \defi{of $\PSL_2$-type} (resp.\ \defi{of $\PGL_2$-type}) if $\tunder$ is projective and for all $\gunder \in T(\tunder)$ the group $\pm \la g_1,g_2,g_3 \ra$ is conjugate to $\PSL_2(k)$ (resp.\ $\PGL_2(k_0)$); by Proposition \ref{MACTHM3}(a), every projective triple is either of $\PSL_2$-type or of $\PGL_2$-type.

We now transfer these results to trace triples in the projective groups $\PSL_2(\F_q)$.  The passage from $\SL_2(\F_q)$ to $\PSL_2(\F_q)$ identifies conjugacy classes whose traces have opposite signs, so associated to a triple of conjugacy classes $\Cunder$ in $\PSL_2(\F_q)$ is a trace triple $(\pm t_1, \pm t_2, \pm t_3)$, which we abbreviate $\pm \tunder$ (remembering that the signs may be taken independently).  We call $\pm \tunder$ a \defi{trace triple up to signs}.

Let $\pm \tunder$ be a trace triple up to signs.  We say $\pm \tunder$ is \defi{commutative} if there exists $\pm \gunder \in T(\pm \tunder)$ such that $\pm \la g_1, g_2, g_3 \ra$ is commutative.  We say $\pm \tunder$ is \defi{exceptional} if there exists a lift of $\pm \tunder$ to a trace triple $\tunder$ such that the associated order triple $(a,b,c)$ is exceptional.  Finally, we say $\pm \tunder$ is \defi{projective} if all lifts $\tunder$ of $\pm \tunder$ are projective, and \defi{partly projective} if there exists a lift $\tunder$ of $\pm \tunder$ that is projective.

\begin{lem}
Every trace triple up to signs is exceptional, commutative, or partly projective.
\end{lem}

\begin{proof}
This follows from Theorem \ref{MACTHM4} and  Lemma \ref{commproj}.
\end{proof}
To a nonsingular trace triple up to signs $\pm \tunder$, we associate the order triple $(a,b,c)$ as the order triple associated to any lift $\tunder$ of $\pm \tunder = (\pm t_1, \pm t_2, \pm t_3)$; this is well defined because we took orders of elements in $\PSL_2(\F_q)$ from the very beginning.  We assume that $a \leq b \leq c$; Remark 1.2 explains why this is no loss of generality.

For a triple of conjugacy classes $\Cunder=(C_1,C_2,C_3)$ of $\PSL_2(\F_q)$, recall we have defined $\Sigma(\Cunder)$ to be the set of generating triples $\gunder=(g_1,g_2,g_3)$ such that $g_i \in C_i$.

\begin{prop} \label{Autactstransitive}
Let $\Cunder$ be a triple of conjugacy classes in $\PSL_2(\F_q)$.
Let $\pm \tunder$ be the associated trace triple up to signs, let $\F_q=\F_p(\pm \tunder)$, and let $(a,b,c)$ the associated order triple.  Suppose that $\pm \tunder$ is partly projective and not exceptional, and let $G=\pm \la g_1,g_2,g_3 \ra \subseteq \PSL_2(\F_q)$.

Then the values $\#\Sigma(\Cunder)/\Inn(G)$ and $\#\Sigma(\Cunder)/\!\Aut(G)$ are given in the following table:

\[ \begin{array}{cccc||c|c}
p & a & abc & G & \#\Sigma(\Cunder)/\Inn(G) & \#\Sigma(\Cunder)/\!\Aut(G) \\
\hline \hline
p=2 & - & - & - & 1 & 1 \\
\hline
\multirow{3}{*}{\textup{$p>2$}} & \multirow{3}{*}{$a=2$} & \textup{$p \mid abc$} & - & 1 & 1 \\
 & & \textup{$p \nmid abc$} & \PGL_2 & 1 & 1 \\
 & & \textup{$p \nmid abc$} & \PSL_2 & 2 & 1 \\
\hline
\multirow{3}{*}{\textup{$p>2$}} & \multirow{3}{*}{$a \neq 2$} & \textup{$p \mid abc$} & - & \leq 2 & \leq 2 \\
 & & \textup{$p \nmid abc$} & \PGL_2 & \leq 2 & \leq 2 \\
& & \textup{$p \nmid abc$} & \PSL_2 & \leq 4 & \leq 2
\end{array} \]
\end{prop}

\begin{proof}
Suppose $p=2$.  Then $\PSL_2(\F_q) = \SL_2(\F_q)$ and by Proposition \ref{MACTHM3}(b) the triple $\Cunder$ is rigid, i.e.\ we have $\#\Sigma(G)/\Inn(G)=\#\Sigma(G)/\Aut(G)=1$.  This gives the first row of the table.
So from now on we suppose $p > 2$.

Let $\tunder=(t_1,t_2,t_3) \in \F_q^3$ be a lift of $\pm \tunder$.  Let $\pm \gunder, \pm \gunder' \in \Sigma(\Cunder)$; lift them to $\gunder,\gunder'$ in $\SL_2(\F_q)^3$ such that $g_1 g_2 g_3 = \pm 1$ and $g_1' g_2' g_3' = \pm 1$ and such that $\tr(g_i)=\tr(g_i')=t_i$.

\emph{Case} 1: \emph{Suppose $a=2$.}  Then $t_1=\tr(g_1)=0=-\tr(g_1)$, so changing the signs of $g_1$ and $g_1'$ if necessary, we may assume that $\gunder,\gunder'$ are triples (that is, $g_1g_2g_3=g_1'g_2'g_3'=1$).  Then by Proposition \ref{MACTHM3}(c), there exists $m \in \SL_2(\overline{\F}_q)$ such that $m$ conjugates $\gunder$ to $\gunder'$.  Since the elements of
$\pm \gunder$ generate $G$ by hypothesis and the elements of $\pm \gunder'$ lie in $G$, it follows that conjugation by $m$ induces an automorphism $\varphi$ of $G$, so $\varphi(\gunder)=\gunder'$, and
$\Cunder$ is weakly rigid.  This gives the entries in the last column for $p>2$ and $a=2$.

\emph{Case} 1(a): \emph{Suppose $a=2$ and $p \mid abc$.}  Then at least one conjugacy class is unipotent and so the two orbits of the set $T(\tunder)$ under $\Out(\PSL_2(\F_q))$ correspond to two different conjugacy class triples and only one belongs to $\Cunder$.  Therefore the triple is in fact rigid.

\emph{Case} 1(b): \emph{Suppose $a=2$ and $p \nmid abc$.}  First, suppose that $G$ is of $\PGL_2$-type.  Then $G\cong\PGL_2(\F_{\sqrt{q}})$ and $\Out(\PGL_2(\F_{\sqrt{q}}))=\la \sigma \ra$; and since $\F_q=\F_p(\tunder)$, the stabilizer of $\la \sigma \ra$ acting on $\tunder$ is trivial as in the analysis following (\ref{outpsl2fq}), and hence the orbits must be already identified by conjugation in $\PGL_2(\F_q)$, so the triple is in fact rigid.

Second, suppose that $G$ is of $\PSL_2$-type.  Then since $\F_q = \F_p(\pm \tunder)$ we have $G=\PSL_2(\F_q)$.  Because $p \nmid abc$, all conjugacy classes $C_i \in \Cunder$ are semisimple and so are preserved under automorphism.  From Proposition \ref{MACTHM3}(b), we see that there are two orbits of $\PSL_2(\F_q)$ acting by conjugation on $\Sigma(\Cunder)$ and the element
$\tau \in \Out(\PSL_2(\F_q))$ induced by conjugation by an element of $\PGL_2(\F_q) \setminus \PSL_2(\F_q)$ identifies these orbits: they are identifed by some element of $\Out(\PSL_2(\F_q))$, but since $\F_q=\F_p(\tunder)$ the stabilizer of $\la \sigma \ra$ acting on $\tunder$ is again trivial.

This completes Case 1 and the table for $p>2$ and $a=2$.  Note that in this case, the choice of the lift $\tunder$ does not figure in the analysis.

\emph{Case} 2: \emph{Suppose $a > 2$.}  Now either $\gunder'$ is already a triple, or changing the sign of $g_1$ we have $\gunder'$ is a triple with trace triple $\tunder'=(-t_1,t_2,t_3)$.  By Lemma \ref{commproj}, this is without loss of generality.

\emph{Case} 2(a): \emph{Suppose $\tunder'=\tunder$}.  Then the same analysis as in Case 1 shows that $\gunder'$ is obtained from $\gunder$ by an automorphism $\varphi$ of $G$, and this automorphism can be taken to be inner except when $p \nmid abc$ and $G=\PSL_2(\F_q)$, in which case up to conjugation there are two triples.

\emph{Case} 2(b): \emph{Suppose $\tunder' \neq \tunder$}.  Then clearly $\gunder'$ is not obtained from $\gunder$ by an inner automorphism.  If $\gunder'=\varphi(\gunder)$ with $\varphi$ an outer automorphism, then after conjugation in $\SL_2(\overline{\F}_q)$, as in Case 2(a), we may assume that $\varphi=\sigma^j$ is a power of the $p$-power Frobenius automorphism $\sigma$, with $\sigma^j(t_1)=-t_1$ and $\sigma^j(t_2)=t_2$ and $\sigma^j(t_3)=t_3$ (with slight abuse of notation).  But again $\F_q$ is generated by the trace triple, so the fixed field $k$ of $\sigma^j$ contains $t_2,t_3$ and $\F_q$ is a quadratic extension of $k$, generated by $t_1$, and $t_1$ is a squareroot from $k$.

This concludes Case 2, with the stated inequalities: we have at most twice as many triples as in Cases 1(a) and 1(b).

We have equality in the first column if and only if $\tunder'$ is projective: otherwise, $\tunder'$ is commutative, as in the proof of Proposition \ref{Autactstransitive}, and all $\gunder' \in \tunder'$ generate affine or commutative subgroups of $\PSL_2(\F_q)$ since they are singular \cite[Theorem 2]{Macbeath}, and any such triple does not belong to $\Sigma(\Cunder)$---the trace triple up to signs is not exceptional so any group generated by a corresponding triple is not projective.  In the second column, we have equality if and only if $\tunder'$ is projective and we are not in the special case described in 2(b).
\end{proof}

\section{Proof of Theorems} \label{sec:proofs}

In this section, we give proofs of the main theorems A, B, and C.

We begin with Theorem A, which follows from the following theorem.

\begin{thm} \label{thmaorders}
Let $(a,b,c)$ be a hyperbolic triple with $a,b,c \in \Z_{\geq 2} \cup \{\infty\}$.  Let $\frakp$ be a prime of $E(a,b,c)$ with residue field $\F_{\frakp}$ lying above the rational prime $p$, and suppose $\frakp \nmid abc$.  If $\frakp \mid 2$,  suppose further that $(a,b,c)$ is not of the form $(mk,m(k+1),mk(k+1))$ with $k,m \in \Z$.  Let $a^\sharp=p$ if $a=\infty$ and $a^\sharp=a$ otherwise, and similarly with $b,c$.

Then there exists a $G$-Galois \Belyi\ map
\[ X(a,b,c;\frakp) \to \PP^1 \]
with ramification indices $(a^\sharp,b^\sharp,c^\sharp)$, where
\[ G=
\begin{cases}
\PSL_2(\F_\frakp), & \text{if $\frakp$ splits completely in $F(a,b,c)$}; \\
\PGL_2(\F_\frakp), &\text{otherwise}.
\end{cases}
\]
\end{thm}

\begin{proof}
Let $(a,b,c)$ be a hyperbolic triple, so that $a,b,c \in \Z_{\geq 2} \cup \{\infty\}$ satisfy $a \leq b \leq c$ and $\chi(a,b,c)=1-1/a-1/b-1/c>0$.  Let $\frakp$ be a prime of the field
\[ E(a,b,c)=\Q(\lambda_a,\lambda_b,\lambda_c,\lambda_{2a}\lambda_{2b}\lambda_{2c}) \]
and let $\frakP$ be a prime of
\[ F(a,b,c)=\Q(\lambda_{2a},\lambda_{2b},\lambda_{2c}) \]
above $\frakp$ above the rational prime $p \nmid abc$.

Then by Proposition \ref{pglyeah}, we have a homomorphism
\[ \phi:\Deltabar(a,b,c) \to \PSL_2(\F_\frakP) \]
with $\tr \phi(\deltabar_s) \equiv \pm \lambda_{2s} \pmod{\frakP}$ for $s=a,b,c$ whose image lies in the subgroup $\PSL_2(\F_\frakP) \cap \PGL_2(\F_\frakp)$.  We have $[\F_\frakP:\F_\frakp] \leq 2$ and
\begin{equation} \label{PGLorPSL}
\PSL_2(\F_\frakP) \cap \PGL_2(\F_\frakp) =
\begin{cases}
\PSL_2(\F_\frakp), & \text{if $\F_\frakP=\F_\frakp$}; \\
\PGL_2(\F_\frakp), & \text{if $[\F_\frakP:\F_\frakp]=2$.}
\end{cases}
\end{equation}

Let $\Deltabar(a,b,c;\frakp)$ be the kernel of the homomorphism $\phi: \Deltabar(a,b,c) \to \PSL_2(\F_\frakP)$.  The generators $\deltabar_s$ of $\Deltabar$ (for $s=a,b,c$) give rise to a triple $\gunder=(g_1,g_2,g_3)$, namely  $g_1=\phi(\deltabar_a)$, $g_2=\phi(\deltabar_b)$, $g_3=\phi(\deltabar_c)$, with trace triple up to signs
\[ \pm \tunder=(\pm t_1, \pm t_2, \pm t_3) \equiv (\pm \lambda_{2a},\pm \lambda_{2b},\pm \lambda_{2c}) \pmod{\frakP}. \]

The map of complex algebraic curves
\[ f:X=X(a,b,c;\frakp)=\Deltabar(a,b,c;\frakp) \backslash \calH \to
\Deltabar(a,b,c) \backslash \calH \cong \PP^1 \]
is a $G$-Galois \Belyi\ map by construction, where $G=\Deltabar(a,b,c)/\Deltabar(a,b,c;\frakp)$.  It is our task to specify $G$ by our understanding of subgroups of $\PSL_2(\F_{\frakP})$.

First, we dispose of the exceptional triples.  Since $(a,b,c)$ is hyperbolic, this leaves the five triples
\[ (a^\sharp,b^\sharp,c^\sharp)=(3,4,4),(2,5,5),(5,5,5),(3,3,5),(3,5,5). \]
Each of these triples is arithmetic, by work of Takeuchi \cite{Takeuchi}, and the result follows from well-known properties of Shimura curves---something that could be made quite explicit in each case, if desired.

Second, we claim that the triple $\pm \tunder$ is not commutative.  From (\ref{commform}), the triple $\tunder$ is commutative if and only if
\[ \beta=\lambda_{2a}^2+\lambda_{2b}^2+\lambda_{2c}^2 - \lambda_{2a}\lambda_{2b}\lambda_{2c}-4 = 0 \]
in $k$.  But $\beta \Z_F$ is the reduced discriminant of the order $\calO$ arising in Section \ref{quatconst}!  And as in the proof of Lemma \ref{pnmid}, from the factorization
\[ \beta=\left(\frac{\zeta_{2b}\zeta_{2c}}{\zeta_{2a}}-1\right)\left(\frac{\zeta_{2a}\zeta_{2c}}{\zeta_{2b}}-1\right)
\left(\frac{\zeta_{2a}\zeta_{2b}}{\zeta_{2c}}-1\right)\left(\frac{1}{\zeta_{2a}\zeta_{2b}\zeta_{2c}}-1\right) \]
we see that the argument in the case $\frakP_K \mid 2$ applies to show that $\beta \not\equiv 0 \pmod{\frakP}$.

Next, we show that orders of $g_1,g_2,g_3$ are $a,b,c$.  Let $s=a,b,c$ and write $g$ for the corresponding element.  We have $\tr \phi(\deltabar_s) \equiv \pm \lambda_{2s} \pmod{\frakP}$.  If $g=1$, then the image is commutative, and this possibility was just ruled out.  We cannot have $g$ unipotent, as then its order $p$ is the prime below $\frakP$, contradicting the hypothesis that $\frakp \nmid abc$.  So we are left with only the possibility that $g$ is a semisimple element of $\PSL_2(\F_{\frakP})$: but then then the order is indeed determined by the trace (see Section \ref{sec:conjclass}), and an element with trace $\lambda_{2s}$ necessarily has order $s$.
The fact that the orders of the images are $a,b,c$ then implies that the ramification indices are $(a,b,c)$.

We continue now with the remainder of the proof of the theorem.  By Proposition \ref{MACTHM4}, we conclude that the triple $\gunder$ is projective.  Then, by Proposition \ref{MACTHM3}(a)--(b), we have either that the image of $\phi$ is either equal to $\PSL_2(\F_\frakP)$, or that the trace triple is irregular and the image of $\phi$ is equal to $\PGL_2(k)$, where $[\F_\frakP:k]=2$.  In the first case, by (\ref{PGLorPSL}) we must have $\F_\frakP=\F_\frakp$ and so the result holds.  In the second case, if $[\F_\frakP:\F_\frakp]=2$ then again by (\ref{PGLorPSL}) the image is already contained in $\PGL_2(\F_\frakp)$ so we must have $k=\F_\frakp$ and again the result holds.

So to conclude, we must rule out the possibility that the trace triple is irregular and that $\F_\frakP=\F_\frakp$.
Since $\F_\frakP=\F_\frakp$, we have
\[ \F_p(t_1,t_2,t_3)=\F_p(t_1^2,t_2^2,t_3^2,t_1t_2t_3). \]
Let $k$ be the subfield of $\F_\frakp$ with $[\F_\frakp:k]=2$.  Then we have
\[ \F_p(\tunder^2)=\F_p(t_1^2,t_2^2,t_3^2) \subseteq k \subseteq \F_p(t_1,t_2,t_3)=\F_\frakp. \]
But we have $[\F_\frakp:\F_p(\tunder^2)] \leq 2$ so $k=\F_p(\tunder^2)$.  If now the triple $\tunder$ is irregular, then without loss of generality (in this argument) we may suppose that $t_1 \in k$ and $t_2,t_3$ are either zero or squareroots of nonsquares in $k$.  But then $t_1t_2t_3 \in k$, so
\[ k=\F_p(\tunder^2)=\F_p(\tunder^2, t_1t_2t_3) = \F_p(t_1,t_2,t_3), \]
a contradiction.
\end{proof}

\begin{cor}
We have
\[ [\Deltabar:\Deltabar(\frakp)]=[\calO_1^{\times}/\{\pm 1\}:\calO_1(\frakP)^{\times}/\{\pm 1\}] \cdot
\begin{cases}
1, & \text{ if $\F_\frakP=\F_\frakp$}; \\
2, & \text{ if $[\F_\frakP:\F_\frakp]=2$}.
\end{cases} \]
\end{cor}

Next, we prove statements (a)--(c) of Theorem B, in two parts.  (In the next section, we will show that the corresponding field extension is unramified away $p$.)  To this end, let $X$ be a curve of genus $g \geq 2$ and let $f:X \to \PP^1$ be a $G$-Galois \Belyi\ map with $G\cong \PGL_2(\F_q)$ or $G \cong \PSL_2(\F_q)$.  Let $(a,b,c)$ be the ramification indices of $f$.

\begin{thm} \label{allbutthma}
Let $r$ be the order of $\Frob_p$ in $\Gal(F_{p'}(a,b,c)/\Q)$.  Then
\[ q = \begin{cases}
\sqrt{p^r}, & \text{ if $G \cong \PGL_2(\F_q)$;} \\
p^r, & \text{ if $G \cong \PSL_2(\F_q)$.}
\end{cases} \]
\end{thm}

\begin{proof}
By work in Section 2, there exists a finite index, normal subgroup $\Gammabar \subseteq \Deltabar(a,b,c)$ such that $\Deltabar(a,b,c)/\Gammabar \cong G$ and the map $X \to \PP^1$ is the map $\Gammabar \backslash \calH \to \Deltabar(a,b,c) \backslash \calH$.  In this way, we have identified the images in $G$ of the monodromy at the three ramification points with the elements $\deltabar_a,\deltabar_b,\deltabar_c$.  Thus, by hypothesis, the images of the triple $(\deltabar_a \Gammabar,\deltabar_b \Gammabar,\deltabar_c \Gammabar)$ in $G$ generates $G$.  This triple lifts to the triple $(-\delta_a \Gammabar,\delta_b \Gammabar,\delta_c \Gammabar)$ in $\SL_2(\F_q)$ with corresponding trace triple
\[ \tunder \equiv (-\lambda_{2a}, \lambda_{2b}, \lambda_{2c}) \pmod{\frakp} \]
for $\frakp$ a prime above $p$ in the field $F_{p'}(a,b,c)$.

If $G \cong \PSL_2(\F_q)$, then Proposition \ref{MACTHM3}(a) implies that $q=\#\F_p(\tunder)=p^r$; this is the residue class field of the prime $\frakp$ above and so $r$ is equal to its residue degree, equal to the order of the Frobenius $\Frob_p$ in the field.  Put another way, $r=\log_p q$ is the least common multiple of the orders of $p$ in $(\Z/2s\Z)^\times/\{\pm 1\}$ for $s=a,b,c$, which is the order of $\Frob_p$ in $\Gal(F_{p'}(a,b,c)/\Q)$, as claimed.  If instead $G \cong \PGL_2(\F_{\sqrt{q}})$, then $r$ is twice this degree.
\end{proof}

Now we prove statements (b)--(c) of Theorem B concerning fields of moduli.

\begin{thm}
The map $f$ is defined over its field of moduli $M(X,f)$ and $M(X,f)$ is an extension of $D_{p'}(a,b,c)^{\la \Frob_p \ra}$ of degree $d_{(X,f)} \leq 2$.  If $a=2$ or $q$ is even, then $d_{(X,f)}=1$.

The map $f$ together with its Galois group $\Gal(f) \cong G$ is defined over its field of moduli $M(X,f,G)$.  Let
\[ D_{p'}(a,b,c)\{\sqrt{p^*}\}=\begin{cases}
D_{p'}(a,b,c)(\sqrt{p^*}), &\text{ if $p \mid abc$, $pr$ is odd, and $G \cong \PSL_2(\F_q)$}; \\
D_{p'}(a,b,c) &\text{ otherwise.}
\end{cases} \]
Then $M(X,f,G)$ is an extension of $D_{p'}(a,b,c)\{\sqrt{p^*}\}$ of degree $d_{(X,f,G)} \leq 2$.  If $q$ is even or $p \mid abc$ or $G \cong \PGL_2(\F_q)$, then $d_{(X,f,G)} =1$.
\end{thm}

\begin{proof}
As in the proof of Theorem \ref{thmaorders}, we may identify the images in $G$ of the monodromy at the three ramification points with the elements $\deltabar_a,\deltabar_b,\deltabar_c \in \Deltabar(a,b,c)$.  Let $\gunder=(\deltabar_a \Gammabar,\deltabar_b \Gammabar,\deltabar_c \Gammabar)$ and let $\Cunder$ be the corresponding conjugacy class triple in $G$.

We refer to the discussion in Section 5, specifically Proposition \ref{WRWR}, and recall the calculation of the field of weak rationality in Section 6 (Lemmas \ref{PGL2fieldofrat} and \ref{PSL2fieldofrat}).  Then by Proposition \ref{Autactstransitive}, we have $d_{(X,f)} = \#\Sigma(\Cunder)/\!\Aut(G) \leq 2$ and $d_{(X,f,G)} = \#\Sigma(\Cunder)/\Inn(G) \leq 4$, with the various cases as listed.

Finally, the map $f$ is defined over its field of moduli by Lemma \ref{canbedefined}.  And we would know that $f$ together with $\Gal(f) \cong G$ is defined over its field of moduli when the centralizer of $G$ in $\Aut X$ is trivial by Lemma \ref{lem:canbedefinedG}.  This holds for any maximal triple (see the discussion surrounding \eqref{fam2}--\eqref{fam1}).  We conclude in fact that it is therefore true for a nonmaximal triple, by the following reasoning.  We need only concern ourselves with the presence of extra automorphisms of the $G$-Galois \Belyi\ map; such an automorphism is given by an element of the maximal group centralizing $G$, and in particular it must normalize the non maximal triangle group.  So we reduce to the case where we have a normal inclusion of triangle groups, and this leaves only the three cases in \eqref{fam1} (the third obtained by a concatenation).  But the desired \Belyi\ map is defined over the field obtained from the maximal triple simply by taking the quotient by the smaller group, and the relevant fields $D_{p'}(a,b,c)^{\langle \Frob_p \rangle})$ and $D_{p'}(a,b,c)\{\sqrt{p^*}\}$ are the \emph{same} in each of these three cases!
\end{proof}

We now prepare for the proof of Theorem C with a discussion of the extension to composite $\frakN$.  Let $\frakN$ be an ideal of $\Z_F$ coprime to $6abc$.  Let $\frakn = \frakN \cap \Z_E$.  Then by Proposition \ref{pglyeah}, we have a homomorphism
\begin{equation} \label{eqn:phifrakN}
\phi_\frakN:\Deltabar \to \PSL_2(\Z_F/\frakN).
\end{equation}
If $\frakM \mid \frakN$ then $\phi_\frakM$ is obtained by the composition of $\phi_\frakN$ with the natural reduction map modulo $\frakM$.  Therefore these maps form a projective system and so we obtain in the limit a map
\[ \widehat{\phi}:\Deltabar \to \prod_{\frakP \nmid 6abc} \PSL_2(\Z_{F,\frakP}) \]
after composing with the Chinese remainder theorem.  The map $\widehat{\phi}$ is injective because $\Deltabar \hookrightarrow \calO_1^\times/\{\pm 1\} \hookrightarrow \PSL_2(\Z_{F,\frakP})$ for any prime $\frakP$.

For a prime $\frakP$ of $\Z_F$ with $\frakp = \frakP \cap \Z_E$ and an integer $e \geq 1$, let $P(\frakP^e) \subseteq \PSL_2(\Z_F/\frakP^e)$ be the group
\[ P(\frakP^e)=
\begin{cases}
\PSL_2(\Z_E/\frakp^e), &\text{ if $\F_\frakP=\F_\frakp$;} \\
\PGL_2(\Z_E/\frakp^e), &\text{ if $[\F_\frakP:\F_\frakp]=2$.}
\end{cases} \]
For an ideal $\frakN$ of $\Z_F$ let $P(\frakN)=\prod_{\frakP^e \parallel \frakN} P(\frakP^e)$, and let $\widehat{P}=\varprojlim_{\frakN} P(\frakN)$ be the projective limit of $P(\frakN)$ with respect to the $\frakN$ for $N \nmid 6abc$.  Then $\widehat{P}$ is a subgroup of $\prod_{\frakP \nmid 6abc} \PSL_2(\Z_{F,\frakP})$.

\begin{prop} \label{densesub}
The image of $\widehat{\phi}$ is dense in $\widehat{P}$.
\end{prop}

We will use the following lemma in the proof.

\begin{lem} \label{lemsurj}
Let $P$ denote either $\PSL_2$ or $\PGL_2$.  Let $H$ be a closed subgroup of $P(\Z_{F,\frakp})$, and suppose that $H$ projects surjectively onto $P(\Z_F/\frakp)$.  If $\opchar R/\frakp \geq 5$, then $H=P(\Z_{F,\frakp})$.
\end{lem}

\begin{proof}
Serre \cite[Lemmas 3.4.2--3.4.3]{SerreAbel} proves this when $\Z_F=\Z$ and $P=\PSL_2$, but his Lie-theoretic proof generalizes to an arbitrary number ring $\Z_F$.  One can deduce the statement for $P=\PGL_2$ from this statement as follows.  The preimage of $H$ under the map $\GL_2(\Z_{F,\frakp}) \to \PGL_2(\Z_{F,\frakp})$ intersected with $\SL_2(\Z_{F,\frakp})$ maps surjectively to $\SL_2(\Z_F/\frakp)$ so $H \supseteq \PSL_2(\Z_{F,\frakp})$.  But
\[ \PGL_2(\Z_{F,\frakp})/\PSL_2(\Z_{F,\frakp}) \cong \Z_{F,\frakp}^\times/\Z_{F,\frakp}^{\times 2} \cong (\Z_F/\frakp)^\times/(\Z_F/\frakp)^{\times 2} \]
and so since $H$ maps surjectively to $\PGL_2(\Z_F/\frakp)$ we must have $H=\PGL_2(\Z_{F,\frakp})$.

An alternate proof for both cases runs as follows.  One proves the statement by induction; we have an exact sequence
\begin{equation} \label{M2GL2}
0 \to M_2(k) \to \GL_2(R/\frakp^e) \to \GL_2(R/\frakp^{e-1}) \to 1
\end{equation}
via the isomorphism
\[ 1+M_2(\frakp^{e-1}/\frakp^e) \cong M_2(\frakp^{e-1}/\frakp^e) \cong M_2(k). \]
The group $\GL_2(R/\frakp^{e-1})$ acts by conjugation on $M_2(k)$ and factors through $\GL_2(k)$.  Since $\opchar k$ is odd, $M_2(k)$ decomposes into irreducible subspaces under this action as $M_2(k)=k \oplus M_2(k)_0$ where $M_2(k)_0$ denotes the subspace of matrices of trace zero.  Restricting to $\SL_2$ or $\PGL_2$, one then reduces to an exercise showing that the above sequence does not split (and indeed, it splits for $k=\F_2$ for $e \leq 3$ and for $k=\F_3$ for $e=2$ \cite[Exercise 1, p.~IV-27]{SerreAbel}).
\end{proof}


\begin{proof}[Proof of Proposition \ref{densesub}]
We show that $\phi_{\frakN}$ has image $P(\frakN)$.  We have shown that this statement is true if $\frakN$ is prime.  Using Lemma
\ref{lemsurj} and basic topological considerations, we find that the image of $\phi_{\frakN}$ is equal to $P(\frakN)$ when $\frakN=\frakP^e$ is a prime power.

Suppose that $\frakM$ and $\frakN$ are coprime ideals of $\Z_F$.  The kernel of the map
\[ \Deltabar \to \frac{\Deltabar}{\Deltabar(\frakM)} \times \frac{\Deltabar}{\Deltabar(\frakN)} \]
is equal to
\[ \Deltabar(\frakM) \cap \Deltabar(\frakN) = \Deltabar \cap (\calO(\frakM)_1^\times \cap \calO(\frakN)_1^\times) = \Deltabar \cap \calO(\frakM\frakN)_1^\times = \Deltabar(\frakM\frakN). \]
The cokernel of this map is $\displaystyle{\frac{\Deltabar}{\Deltabar(\frakM)\Deltabar(\frakN)}}$.  We claim that this cokernel is trivial.  Since $\calO$ is dense in $\calO_\frakN = \calO \otimes_{\Z_F} \Z_{F,\frakN}$ it follows that $\calO(\frakM)$ is dense in $\calO_\frakN$, so $\calO(\frakM)_1^\times$ maps surjectively modulo $\frakN$ onto $\calO_1^\times/\calO(\frakN)_1^\times \cong \SL_2(\Z_F/\frakN)$.  Thus $\calO(\frakM)_1^\times \calO(\frakN)_1^\times = \calO_1^\times$, and so $\Deltabar(\frakM)\Deltabar(\frakN)=\Deltabar$.   Composing with the map $(\phi_\frakM,\phi_\frakN)$, we obtain a map
\[ \Deltabar \to \PSL_2(\Z_F/\frakM) \times \PSL_2(\Z_F/\frakN); \]
by induction on the number of prime factors, we may suppose that the image of the this map is equal to $P(\frakM) \times P(\frakN) \cong P(\frakM\frakN)$, and the result follows.
\end{proof}

Having now defined the curve $X(a,b,c;\frakN)=X(\frakN) = \Deltabar(\frakN) \backslash \calH$ and identified its Galois group as a cover of $X(1) = \Deltabar \backslash \calH$, to conclude this section we also define certain intermediate quotients in analogy with the classical modular curves.  Recall by \eqref{eqn:phifrakN} we have a homomorphism $\phi_{\frakN}:\Deltabar \to \PSL_2(\Z_F/\frakN)=P$; we let $H_0 \leq P$ be the image in $P$ of subgroup of upper-triangular matrices in $\SL_2(\Z_F/\frakN)$ and similarly $H_1 \leq P$ the image of the subgroup of upper-trianglar matrices with both diagonal entries equal to $1$.  Let
\[ \Delta(\frakN) \leq \Delta_1(\frakN)=\phi_{\frakN}^{-1}(H_1) \leq \Delta_0(\frakN)=\phi_{\frakN}^{-1}(H_0) \leq \Deltabar, \]
be the preimages of $H_0$ and $H_1$ under $\phi_{\frakN}$; then we define
\[ X_0(\frakN)=\Delta_0(\frakN) \backslash \calH, \qquad X_1(\frakN)=\Delta_1(\frakN) \backslash \calH \]
and we obtain maps
\begin{equation} \label{eqn:towerXN}
X(\frakN) \to X_1(\frakN) \to X_0(\frakN) \to X(1).
\end{equation}
Although the definition of these curves depends on a choice of isomorphism $\phi_{\frakN}$, any two such choices are conjugate in $G=\Aut(X_0(\frakN)/X(1))$, so the resulting tower \eqref{eqn:towerXN} does not depend on any choices up to isomorphism.
The curves $X_0(\frakN),X_1(\frakN)$ are defined over any field of definition of $(X(\frakN),G)$.

\section{Examples}

In this section, we give many examples of Theorems A and B and show how these theorems recover some familiar families of curves.
\\ \\
Our notation $X(a,b,c;\mathfrak{p})$ becomes awkward in the consideration of specific cases: given $a,b,c \in \Z_{\geq 2}$
and a prime number $p \nmid 2abc$, we need to choose a prime ideal $\mathfrak{p}$ of $E(a,b,c)$ lying over $p$.  By Theorem B, the curve $X(a,b,c;\mathfrak{p})$ only depends upon the choice of $\mathfrak{p}$ lying over $p$ up to Galois
conjugacy, and there is in general no canonical choice of $\mathfrak{p}$.  Therefore in what follows we will write
$X(a,b,c;p)$ for $X(a,b,c;\mathfrak{p})$ for some unspecified $\mathfrak{p} \mid p$.

\begin{exm} \label{modularfielddef}
 We take $(a,b,c)=(2,3,\infty)$ as $\SL_2(\Z) \cong \Delta(2,3,\infty)$.  For $N \in \Z_{\geq 1}$, our construction from triangle groups gives exactly the congruence subgroup $\Delta(2,3,\infty;N)=\Gamma(N)$ of matrices congruent to the identity modulo $N$, and we find the modular curve $X(2,3,\infty;N)=X(N)=\Gamma(N) \backslash \calH^*$ of level $N$, where $\calH^*=\calH \cup \PP^1(\Q)$ is the completed upper half-plane.
Suppose $N=p$ is prime.  The cover $X(p) \to X(1)$ is a $G=\PSL_2(\F_p)$-cover with ramification indices $(2,3,p)$; this verifies the statement of Theorem B, since $F=F(a,b,c)=\Q(\lambda_4,\lambda_6,\lambda_\infty)=\Q=E$.  The statement of Theorem A says that $M(X(p))=\Q$ since $a=2$.  As discussed in Example \ref{XpoverQ}, in fact we have $M(X(p),G)=\Q(\sqrt{p^*})$.
\end{exm}

\begin{exm} \label{23pp}
We take $(a,b,c)=(2,3,p)$ with $p \geq 7$ prime.  We want to consider the case where $q$ is a power of $p$; for this, we will need the slight extension of Theorem B given in Remark \ref{TheoremBext}.  We have $F=F(\lambda_4,\lambda_6,\lambda_{2p})=F(\lambda_{2p})=E$;  the discriminant
\[ \beta=\lambda_{4}^2+\lambda_{6}^2+\lambda_{2p}^2 + \lambda_{4}\lambda_{6}\lambda_{2p} - 4 = \lambda_{2p}^2-3 = \lambda_p-1 \]
is a unit in $\Z_{E}$.  Therefore we can consider $X(2,3,p;\frakp)$ where $\frakp^{(p-1)/2}=(p)$, which gives a $\PSL_2(\F_p)$-cover $X(2,3,p;\frakp) \to X(2,3,p) \cong \PP^1$.

We verify Theorem A: we have $\Q(\lambda_{4},\lambda_{6},\lambda_{2p})_{p'} = \Q$ so $r=1$.  Since $a=2$ and $p \mid abc$, we have $d_X=d_{(X,G)}=1$; consequently, $M(X)=\Q$ and $M(X,G)=\Q(\sqrt{p^*})$.

The proof in this situation comes down to the following.  There are two unipotent conjugacy classes which are in the same Galois orbit (taking an odd power moves from quadratic residues to nonresidues) and so the field of rationality of such a conjugacy class is the quadratic subfield $\Q(\sqrt{p^*}) \subseteq \Q(\zeta_p)$, where $p^*=(-1)^{(p-1)/2} p$.  Since the other two conjugacy classes representing elements of orders $2$ and $3$ are $\Q$-rational, the field of rationality of $\Cunder$ is $F(\Cunder)=\Q(\sqrt{p^*})$.  As above, the outer automorphism $\tau$ interchanges the two unipotent conjugacy classes, so $F\sbwk(\Cunder)=\Q$.

In fact, the classical modular cover $j:X(p) \to X(1)$ is also a $\PSL_2(\F_p)$-Galois \Belyi\ map, with ramification points  $0,1728,\infty$.  It follows that $X(2,3,p;p) \cong X(p)$ over $\Qbar$.

We could equally well consider the covers $X=X(2,3,p;2) \to X(2,3,p)$, which for the same reasons can be defined over $\Q$ and gives rise to $\SL_2(\F_{2^r})=\PSL_2(\F_{2^r})=\PSL_2(\F_{2^r})$ covers where $r$ is the order of $2$ in $(\Z/p\Z)^\times/\{\pm 1\}$; from Theorem A, we have $M(X,G)=\Q(\lambda_p)$.
\end{exm}

\begin{exm} \label{Shimurafielddef}
Finitely many families of curves $X(a,b,c;p)$ correspond to Shimura curves, where the group $\Delta(a,b,c)$ is arithmetic.  The arithmetic triples were classified by Takeuchi in a sequence of papers \cite{Takeuchi00,Takeuchi0,Takeuchi,Takeuchi2}: there are $85$ triples $(a,b,c)$ and $26$ maximal triples which fall into $19$ commensurability classes.  As Takeuchi explains (in the notation of Section \ref{quatconst}), the triple $(a,b,c)$ gives rise to arithmetic groups if and only if the quaternion algebra $A$ over $E$ is split at exactly one real place of $E$, and the triangle group $\Delta(a,b,c)$ is commensurable with the unit group of a maximal order in $A$.

This covers as a special case the previous example of the modular curves, which include the maximal triples $(2,3,\infty),(2,4,\infty),(2,6,\infty)$ and the nonmaximal triples
\[ (3,3,\infty),(3,\infty,\infty),(4,4,\infty),(6,6,\infty),(\infty,\infty,\infty). \]
So in this example, we restrict to the case $a,b,c \in \Z$.

The minimal field of definition of these curves was studied by Elkies \cite[\S 5.3]{ElkiesSCC} and later by the second author \cite[Proposition 5.1.2]{VoightThesis}.  Each base field $E$ is Galois over $\Q$, and it turns out that at least one of the triangle groups in each commensurability class has distinct indices $a,b,c$: therefore, by identifying the corresponding elliptic points with $0,1,\infty$, any Galois-invariant construction (such as taking Galois invariant level) yields a curve which is fixed by $\Gal(E/\Q)$, providing extra descent in some cases.

This argument can be made directly using the language of canonical models of Shimura curves, which has the advantage that it applies in other circumstances as well (see e.g.\ Hallouin \cite[Proposition 1]{Hallouin}).  Let $B$ be a quaternion algebra over a totally real field $F$ which is split at a unique real place and let $\calO$ be a maximal order in $B$.  Associated to this data is a Shimura curve $X(\C)=B_+^\times \backslash \calH \times \Bhat^\times / \calOhat^\times$ which has a model $X$ over the reflex field $F^\flat=F$.  Suppose $F$ has strict class number $1$, so $X(\C)$ is irreducible (otherwise consider a component over the strict class field of $F$).  Suppose further that $F$ is Galois over $\Q$.  Then for any $\sigma \in \Gal(F/\Q)$, the conjugate curve $X^\sigma$ is given by
\[ X^\sigma(\C) = (B^\sigma)_+^\times \backslash \calH \times (\Bhat^\sigma)^\times / \calOhat^\sigma)^\times. \]
Now if $B^\sigma \cong_{\Q} B$, which means precisely that the discriminant of $B$ is invariant under $\sigma$, and there exists an analytic isomorphism $X(\C) = \Gamma(1) \backslash \calH \xrightarrow{\sim} \Gamma(1)^\sigma \backslash \calH = X^\sigma(\C)$, then this yields exactly the descent data needed to descend $X$ to $\Q$.  In the case of triangle groups, the quotients $X^\sigma(\C)$ have fundamental domain given by the union of two hyperbolic triangles with angles $\pi/a,\pi/b,\pi/c$, and any two hyperbolic triangles with the same angles are congruent; hence the curves descend.  This argument works with the maximal order replaced by any order $\calO$ which is defined by Galois invariant means, e.g.\ the order $\calO(N)$ for $N \geq 1$.

Many of these triples will occur in specific examples below.
\end{exm}

\begin{exm} \label{exmLLT}
Consider the case of Hecke triangle groups treated by Lang, Lim, and Tan \cite{LangLimTan}, the groups with $\Deltabar(a,b,c)=\Deltabar(2,b,\infty)$.  We make the additional assumption that $b$ is odd. Then we have
\[ F(4,2b,\infty)=\Q(\lambda_4,\lambda_{2b},\lambda_{\infty})=\Q(\lambda_{2b})=\Q(\lambda_q)=E(2,b,\infty) \]
since $b$ is odd.  Then for all primes $\frakp$ of $\Q(\lambda_q)$ we have $\F_\frakP=\F_\frakp$ in the notation of Section \ref{quatconst}, so $\Deltabar/\Deltabar(\frakp) \cong \PSL_2(\F_\frakp)$.  Note that when $q \neq p$ we have that $[\F_\frakp:\F_p]$ is indeed equal to the smallest positive integer $r$ such that $p^r \equiv \pm 1 \pmod{q}$, or equivalently the order of $\Frob_p$ in $\Gal(\Q(\lambda_q)/\Q)$.

Lang, Lim, and Tan \cite[Main Theorem, part (iii)]{LangLimTan} obtain a group of $\PGL_2$-type in the case that $r$ is even and $[\F_p(t):\F_p(t^2)]=2$, where $t \equiv \lambda_q \pmod{\frakp}$.  But this latter equality cannot hold: the map $\zeta_q \mapsto \zeta_q^2$ is a Galois automorphism of $\Q(\zeta_q)$ and restricts to the automorphism $\lambda_q \mapsto \lambda_q^2-2$.  It follows in fact that $\F_p(t)=\F_p(t^2)$.
\end{exm}

To give further examples, we list all $G$-Galois \Belyi\ curves with $a,b,c \neq \infty$ with genus $g \leq 24$ with $G=\PSL_2(\F_q)$ or $G=\PGL_2(\F_q)$. The formula for the genus (Remark \ref{onlyfinGaloisBelyi}) gives a bound for $a,b,c$ and $\#G$ in terms of $g$ (in fact, for arbitrary groups $G$).  From the bound $\#\PSL_2(\F_q) \leq 84(g-1) \leq 1932$ we obtain $q \leq 16$; and for each group $G$ we find only finitely many triples of possible orders $(a,b,c)$.  We then enumerate the triples in \textsc{Magma} \cite{Magma}; the curves of genus $g \leq 24$ are listed in Table \ref{tablelowgenus}.
We list also if the triple $(a,b,c)$ is arithmetic (after Takeuchi) and the genus $g_0$ of the subcover $X_0(a,b,c;\frakp)$ whose Galois closure is $X(a,b,c;\frakp)$, as defined at the end of Section \ref{sec:proofs}; we also give the corresponding fields $F(a,b,c) \supseteq E(a,b,c)$, defined in (*).

\begin{table}
\begin{equation} \label{tablelowgenus} \notag
\begin{array}{c|c|c||c|c||c|c|c}
g & (a,b,c) & G & \text{arithmetic?} & g_0 & F(a,b,c) & E(a,b,c) & D_{p'}(a,b,c)^{\langle \Frob_p \rangle} \\
\hline \hline
 3\rule{0pt}{2.5ex} & (2,3,7) & \PSL_2(\F_7) & \textsf{T} & 0 & \Q(\lambda_7) & \Q(\lambda_7) & \Q \\
 3 & (3,4,4) & \PGL_2(\F_3) & \textsf{T} & 1 & \Q(\sqrt{2}) & \Q & \Q \\
 4 & ( 2, 4, 5 ) & \PGL_2(\F_5) & \textsf{T} & 0 & \Q(\sqrt{2},\sqrt{5}) & \Q(\sqrt{5}) & \Q \\
 4 & ( 2, 5, 5 ) & \PGL_2(\F_4) & \textsf{T} & 1 & \Q(\sqrt{5}) & \Q(\sqrt{5}) & \Q \\
 4 & ( 2, 5, 5 ) & \PSL_2(\F_5) & \textsf{T} & 0 & \Q(\sqrt{5}) & \Q(\sqrt{5}) & \Q \\
 5 & ( 3, 3, 5 ) & \PGL_2(\F_4) & \textsf{T} & 0 & \Q(\sqrt{5}) & \Q(\sqrt{5}) & \Q \\
 5 & ( 3, 3, 5 ) & \PSL_2(\F_5) & \textsf{T} & 1 & \Q(\sqrt{5}) & \Q(\sqrt{5}) & \Q \\
 6 & ( 2, 4, 6 ) & \PGL_2(\F_5) & \textsf{T} & 0 & \Q(\sqrt{2},\sqrt{3}) & \Q & \Q \\
 7 & ( 2, 3, 7 ) & \PGL_2(\F_8) & \textsf{T} & 0 & \Q(\lambda_7) & \Q(\lambda_7) & \Q \\
 8 & ( 2, 3, 8 ) & \PGL_2(\F_7) & \textsf{T} & 0 & \Q(\lambda_{16}) & \Q(\sqrt{2}) & \Q \\
 8 & ( 3, 3, 4 ) & \PSL_2(\F_7) & \textsf{T} & 0 & \Q(\sqrt{2}) & \Q(\sqrt{2}) & \Q(\sqrt{2}) \\
 9 & ( 2, 5, 6 ) & \PGL_2(\F_5) & \textsf{T} & 1 & \Q(\sqrt{3},\sqrt{5}) & \Q(\sqrt{5}) & \Q \\
 9 & ( 3, 5, 5 ) & \PGL_2(\F_4) & \textsf{T} & 1 & \Q(\sqrt{5}) & \Q(\sqrt{5}) & \Q \\
 9 & ( 3, 5, 5 ) & \PSL_2(\F_5) & \textsf{T} & 1 & \Q(\sqrt{5}) & \Q(\sqrt{5}) & \Q \\
 10 & ( 2, 4, 5 ) & \PSL_2(\F_9) & \textsf{T} & 0 & \Q(\sqrt{2},\sqrt{5}) & \Q(\sqrt{5}) & \Q \\
 10 & ( 2, 4, 7 ) & \PSL_2(\F_7) & \textsf{T} & 1 & \Q(\lambda_7,\sqrt{2}) & \Q(\lambda_7) & \Q \\
 11 & ( 2, 6, 6 ) & \PGL_2(\F_5) & \textsf{T} & 1 & \Q(\sqrt{3}) & \Q & \Q \\
 11 & ( 3, 4, 4 ) & \PGL_2(\F_5) & \textsf{T} & 0 & \Q(\sqrt{2}) & \Q & \Q \\
 13 & ( 5, 5, 5 ) & \PGL_2(\F_4) & \textsf{T} & 2 & \Q(\sqrt{5}) & \Q(\sqrt{5}) & \Q \\
 13 & ( 5, 5, 5 ) & \PSL_2(\F_5) & \textsf{T} & 1 & \Q(\sqrt{5}) & \Q(\sqrt{5}) & \Q \\
 14 & ( 2, 3, 7 ) & \PSL_2(\F_{13}) & \textsf{T} & 0 & \Q(\lambda_7) & \Q(\lambda_7) & \Q(\lambda_7) \\
 15 & ( 2, 3, 9 ) & \PGL_2(\F_8) & \textsf{T} & 1 & \Q(\lambda_9) & \Q(\lambda_9) & \Q \\
 15 & ( 2, 4, 6 ) & \PGL_2(\F_7) & \textsf{T} & 0 & \Q(\sqrt{2},\sqrt{3}) & \Q & \Q \\
 15 & ( 3, 4, 4 ) & \PSL_2(\F_7) & \textsf{T} & 1 & \Q(\sqrt{2}) & \Q & \Q \\
 16 & ( 2, 3, 8 ) & \PGL_2(\F_9) & \textsf{T} & 0 & \Q(\lambda_{16}) & \Q(\sqrt{2}) & \Q \\
 16 & ( 3, 3, 4 ) & \PSL_2(\F_9) & \textsf{T} & 0 & \Q(\sqrt{2}) & \Q(\sqrt{2}) & \Q \\
 16 & ( 3, 4, 6 ) & \PGL_2(\F_5) & \textsf{T} & 1 & \Q(\sqrt{2},\sqrt{3}) & \Q(\sqrt{6}) & \Q \\
 17 & ( 3, 3, 7 ) & \PSL_2(\F_7) & \textsf{T} & 0 & \Q(\lambda_7) & \Q(\lambda_7) & \Q \\
 19 & ( 2, 7, 7 ) & \PSL_2(\F_7) & \textsf{T} & 0 & \Q(\lambda_7) & \Q(\lambda_7) & \Q \\
 19 & ( 2, 5, 5 ) & \PSL_2(\F_9) & \textsf{T} & 1 & \Q(\sqrt{5}) & \Q(\sqrt{5}) & \Q \\
 19 & ( 4, 4, 5 ) & \PGL_2(\F_5) & \textsf{T} & 0 & \Q(\sqrt{2},\sqrt{5}) & \Q(\sqrt{5}) & \Q \\
 21 & ( 3, 6, 6 ) & \PGL_2(\F_5) & \textsf{T} & 2 & \Q(\sqrt{3}) & \Q & \Q \\
 22 & ( 2, 4, 8 ) & \PGL_2(\F_7) & \textsf{T} & 1 & \Q(\lambda_{16}) & \Q(\sqrt{2}) & \Q(\sqrt{2}) \\
 22 & ( 4, 4, 4 ) & \PSL_2(\F_7) & \textsf{T} & 2 & \Q(\sqrt{2}) & \Q(\sqrt{2}) & \Q \\
 24 & ( 3, 4, 7 ) & \PSL_2(\F_7) & \textsf{F} & 1 & \Q(\lambda_7,\sqrt{2}) & \Q(\lambda_7,\sqrt{2}) & \Q \\
 24 & ( 4, 5, 6 ) & \PGL_2(\F_5) & \textsf{F} & 1 & \Q(\sqrt{2},\sqrt{3},\sqrt{5}) & \Q(\sqrt{30}) & \Q
\end{array}
\end{equation}
\defi{Table \ref{tablelowgenus}}: $\PSL_2(\F_q)$-Galois \Belyi\ curves of genus $g \leq 24$
\end{table}
\addtocounter{equation}{1}

We now discuss many of these curves in turn, highlighting those curves with interesting features.  We note first that all curves but the last two (of genus $24$) are arithmetic.  Our computations are performed in \textsc{Magma} \cite{Magma}.  We are grateful to Elkies for allowing us to record several computations and other remarks in these examples.

\subsubsection*{Genus 3, $(2,3,7)$, $\PSL_2(\F_7)$}

This curve is the beloved Klein quartic, the projective plane curve given by the equation $x^3y + y^3z + z^3x=0$.  For a detailed discussion of this curve and its arithmetic, see Elkies \cite{Elkies}.  The map $f:X_0(2,3,7;7) \cong \PP^1 \to X(2,3,7)=\PP^1$ is given by the rational function
\[ \frac{(t^4+14t^3+63t^2+70t-7)^2}{1728t}=\frac{(t^2+5t+1)^3(t^2+13t+49)}{1728t}+1. \]
This verifies in Theorem A that the curve is defined over $\Q$ (since $a=2$) and its automorphism group is defined over $\Q(\sqrt{-7})$.

\subsubsection*{Genus 3, $(3,4,4)$, $\PGL_2(\F_3)$}

The triple $(3,4,4)$ is exceptional, and so has been excluded from our analysis.  But since $\PGL_2(\F_3) \cong S_4$ is the full group, it is worth identifying this cover.  The quaternion algebra $A$ associated to the triple $(3,4,4)$ is defined over $E=\Q(\lambda_3,\lambda_4,\lambda_6\lambda_8^2)=\Q$ and has discriminant $6$.  The group $\Delta(3,4,4)$ is not maximal: it is contained in $\Delta(2,4,6)$ with index $2$.  The curve $X(2,4,6)$ is associated to the arithmetic group $N(\Lambda)/\Q^\times$ for $\Lambda \subseteq A$ a maximal order, and the quotient $X(3,4,4) \to X(2,4,6)$ is obtained as the quotient by an Atkin-Lehner involution; the Shimura curve associated to a maximal order has signature $(0;2,2,3,3)$.  See work of Baba and Granath \cite{BabaGranath} as well as Elkies \cite[\S 3.1]{ElkiesSCC} for a detailed discussion of these triangle groups and their relationships.

Theorem B does not apply nor does Remark \ref{TheoremBext} since $3$ is ramified in this quaternion algebra.  Consequently, the algebra $A \otimes_\Q \Q_3$ is a division algebra and $\Lambda \otimes_\Z \Z_3$ is the unique maximal order.  There is a unique two-sided prime ideal $P \subset \Lambda$ with $\nrd(P)=3$.  The quotient $\Lambda/P$ is isomorphic to $\F_9=\F_3(i)$, and $\Lambda/P^2=\Lambda/3\Lambda$ is isomorphic to the algebra over $\F_3$ generated by $i,j$ subject to $i^2=-1$, $j^2=0$, and $ji=-ij$.  We have an exact sequence
\[ 1 \to (1+P)/(1+3\Lambda) \to (\Lambda/3\Lambda)^\times/\{\pm 1\} \to (\Lambda/P)^\times/\{\pm 1\} \to 1 \]
which as finite groups is
\[ 1 \to \Z/3\Z \oplus \Z/3\Z \to (\Lambda/3\Lambda)^\times/\{\pm 1\} \to \Z/4\Z \to 1. \]
There is a dicyclic group of order 12, a subgroup of $(\Lambda/3\Lambda)^\times/\{\pm 1\}$ that maps
surjectively onto $\Z/4\Z$ with kernel $\Z/3\Z$; this gives a cover $X(3,4,4;3) \to X(\Lambda)$ of degree $12$.  There is an Atkin-Lehner involution (normalizing $\Lambda$) whose quotient then yields the cover $f:X(3,4,4;3) \to X(3,4,4)$ with Galois group $\PGL_2(\F_3)$.

The conjugacy classes of elements of orders $3$ and $4$ in $S_4$ are unique, and one can check by hand that the corresponding triple is rigid, so the curve $X(3,4,4;3)$ (with its automorphism group) is defined over $\Q$.

Wolfart \cite[\S 6.3]{WolfartCMJac} identifies $X(3,4,4;3)$ as the hyperelliptic curve
\[ y^2 = x^8 - 14x^4 + 1=(x^4 - 4x^2 + 1)(x^4 + 4x^2 + 1) \]
with automorphism group $S_4 \times C_2$.  The roots of the polynomial in $x$ are the vertices of a cube in the Riemann sphere.  The genus $1$ curve $X_0(3,4,4;3)$ is a degree $4$ cover of $X(3,4,4)$ and corresponds to the fixed field under a subgroup $S_3 \subseteq S_4$: it is the elliptic curve with minimal model
\[ y^2=x^3+x^2+16x+180 \]
of conductor $48$ and the map $X_0(3,4,4;3) \to X(3,4,4)$ is the map
\[ \phi(x,y) = 56 + x^2 - 4y \]
with divisor $4(4,18) - 4\infty$ and the divisor of $\phi-108$ is $3(-2,-12) + (22,108) - 4\infty$.  Via $\phi(x,y)=t$, this cover gives rise to the family of $S_4$-extensions
\[ (x^2+56-t)^2 - 16(x^3+x^2+16x+180) = x^4 - 16x^3 + (96-2t)x^2 - 256x + (t^2 - 112t + 256). \]

This shows that for exceptional (or commutative) triples there may be normal subgroups of $\Deltabar(a,b,c)$ with quotient isomorphic to $\PSL_2(\F_q)$ or $\PGL_2(\F_q)$ which are not obtained by the construction of Theorem A.  In general, covers obtained by considering suborders of index supported at primes dividing the discriminant of the quaternion algebra will give only solvable extensions.

\subsubsection*{Genus 4, $(2,4,5)$, $\PGL_2(\F_5)$; $(2,5,5)$, $\PGL_2(\F_4) = \PSL_2(\F_5)$}

The triangle group $\Deltabar(2,4,5)$ is maximal, associated to a
quaternion algebra defined over $\Q(\sqrt{5})$ ramified at the prime $(2)$
(obtained as the full Atkin-Lehner quotient), and this group contains $\Deltabar(2,5,5)$ as a subgroup of index 2.  We have an exceptional isomorphism $\PGL_2(\F_4) = \PSL_2(\F_4) \cong \PSL_2(\F_5)$, so this curve arises from the congruence subgroup with $\frakp=(\sqrt{5})$.  The triple $(2,5,5)$ is exceptional, but the spherical triangle group that it generates is the full group ($\PSL_2(\F_5) \cong A_5$; note $\PGL_2(\F_5) \cong S_5$).

The curve is the Bring curve (see Wolfart \cite[\S 6.4]{WolfartCMJac} and also Edge \cite{Edge}), defined by the equations
\[  x_0 + x_1 + ... + x_4 = x_0^2 + x_1^2 + ... + x_4^2 = x_0^3 + x_1^3 + ... + x_4^3 = 0 \]
in $\PP^4$.

The significance in Theorem B about the splitting behavior of primes is illustrated here.  We have $F_5(2,4,5)=\Q(\sqrt{2})$ and $F_5(2,5,5)=\Q$, whereas the trace field of the square subgroup is $E_5(2,4,5)=\Q=E_5(2,5,5)=E$; the prime $5$ is inert in $\Q(\sqrt{2})$, so in the former case we obtain a $\PGL_2(\F_5)$-extension and in the latter we obtain a $\PSL_2(\F_5)$-extension.

\subsubsection*{Genus 5, $(3,3,5)$, $\PGL_2(\F_4) \cong \PSL_2(\F_5)$}

The triple $(3,3,5)$ is exceptional but again generates the full group.  However, the quaternion algebra $A$ is defined over $\Q(\sqrt{5})$ and is ramified at $(\sqrt{5})$, and $\Deltabar(3,3,5)$ corresponds to the group of units of reduced norm 1 in a maximal order.  (The Atkin-Lehner quotient is uniformized by the triangle group $\Deltabar(2,3,10)$---the composite gives a $G$-Galois \Belyi\ map, but $G \not\cong \PGL_2(\F_5)$, since there is no element of order $10$ in this group!)

The conjugacy class of order $3$ is unique and there are two of order $5$; we check directly that this cover is rigid, even though it is exceptional.  Therefore the curve $X=X(3,3,5;5)$ is defined over $\Q$ and its automorphism group is defined over $\Q(\sqrt{5})$.

The cover $X_0(3,3,5;4)$ has genus $0$, and is given by
\[ t^3(6t^2-15t+10) - 1 = (t-1)^3(6t^2+3t+1). \]


\subsubsection*{Genus 6, $(2,4,6)$, $\PGL_2(\F_5)$}

As in the $(3,4,4;3)$ case above, the curve $X(2,4,6)$ is the full $(\Z/2\Z \times \Z/2\Z)$-Atkin-Lehner quotient of the curve corresponding to the discriminant 6 quaternion algebra over $\Q$.  The curve $X = X(2,4,6;5)$ is rather exotic, in that it does not arise from a congruence subgroup in the usual sense: the usual (Shimura curve) congruence subgroup of level $5$ gives a $\PSL_2(\F_5)$-cover of genus $11$ mapping to a conic $X(1)$ defined by $x^2+3y^2+z^2=0$ and its quotients by Atkin-Lehner involutions give $\PGL_2(\F_5)$-covers, as below.  In particular, the curve $X(1)$ does not occur intermediate to $X(2,4,6;5) \to X(2,4,6)$.

Here in Theorem B we have $F(2,4,6)=\Q(\sqrt{2},\sqrt{3})$ and $E(2,4,6)=\Q$; the prime $5$ has inertial degree $2$ in this extension, hence we get $\PGL_2(\F_5)$.  In Theorem A we have $a=2$ and $G \cong \PGL_2(\F_5)$ so $d_X=d_{(X,G)}=1$, hence $M(X)=M(X,G)=\Q$.

The map $X_0(2,4,6;5) \to X(2,4,6)$ is computed by Elkies \cite{ElkiesSCC}:
\[ (540t^6 + 324t^5 + 135t^4 + 1) - 1 = 27t^4(20t^2+12t+5). \]
And as Elkies observes, by the invariant theory of $\PGL_2(\F_5)$, there are invariants of degree $12$, $20$, and $30$, and a relation of degree $60$, so the cover can be written invariantly as
\[ y^2 = F_{12}^5/F_{30}^2 \]
and this gives the quotient.

\subsubsection*{Genus 7, $(2,3,7)$, $\PGL_2(\F_8) (=\SL_2(\F_8))$}

The curve $X=X(2,3,7;2)$ is the Fricke-Macbeath curve \cite{Macbeath2} of genus $7$, the second smallest genus for a \defi{Hurwitz curve}: i.e., a curve
uniformized by a subgroup of the Hurwitz group $\Delta(2,3,7)$ (and therefore having maximal automorphism group for its genus).  The curve $X$ has field of moduli equal to $\Q$ and the minimal field of definition of $(X,G)$ is $\Q(\lambda_7)$.  Berry and Tretkoff \cite{BT} show that the Jacobian $J$ of $X$ is isogenous to $E^7$, where $E$ is a non-CM elliptic curve with rational $j$-invariant.  (See also Wolfart \cite[\S 6.5]{WolfartCMJac} and Wohlfahrt \cite{Wohlfahrt}.)

\subsubsection*{Genus 8, $(2,3,8)$, $\PGL_2(\F_7)$ and $(3,3,4)$, $\PSL_2(\F_7)$}

The curve $X_0(2,3,8;7)$ has genus 0, and the triangle group $\Deltabar(2,3,8)$ arises from the quaternion algebra over $\Q(\sqrt{2})$ ramified at the prime above $2$: the map is
\begin{align*}
\phi(t) &= t^8 + \frac17(-4\sqrt{2} - 16)t^7 + (-4\sqrt{2} + 6)t^6 + 6\sqrt{2}t^5 + \frac12(-36\sqrt{2} + 39)t^4 \\
&\qquad +    (3\sqrt{2} - 12)t^3 + \frac12(-46\sqrt{2} + 79)t^2 + \frac12(9\sqrt{2} - 8)t + \frac{1}{16}(-248\sqrt{2} +
    313)
\end{align*}
which factors as
\[ \phi(t)=\left(t^2 + \frac12(-2\sqrt{2} + 1)\right)^3\left(t^2 + \frac17(-4\sqrt{2} - 16)t + \frac12(-2\sqrt{2} + 9)\right) \]
and
\begin{align*}
\phi(t)-\frac{1}{27}(4\sqrt{2}+5) &= \left(t^2 + \frac17(3\sqrt{2} + 12)t + \frac{1}{14}(-8\sqrt{2} + 31)\right) \\
&\qquad \cdot \left(t^3 + \frac12(-\sqrt{2} - 4)t^2 + \frac12(-2\sqrt{2} + 7)t + \frac14\sqrt{2}\right)^2.
\end{align*}

We have $\Delta(3,3,4) \subseteq \Delta(2,3,8)$ with index $2$ and given by the quotient of the Atkin-Lehner involution, so this gives us also a $\PSL_2(\F_7)$-subcover.

Since $a=2$ we have $d_X=1$ and so $M(X)=\Q(\lambda_8)^{\la \Frob_7 \ra}=\Q(\sqrt{2})$, in agreement.  We only know $d_{(X,G)} \leq 2$, so this is the first example of a curve where the automorphism group may be defined over a quadratic extension of $\Q(\sqrt{2})$.

\subsubsection*{Genus 14, $(2,3,7)$, $\PSL_2(\F_{13})$}

The curve $X(2,3,7;13)$ is a Hurwitz curve.  Some progress has been made in writing down equations for this curve: see work by Moreno-Mej\'ia \cite{Moreno} (and work in progress by Streit; methods of Streit \cite{HBCStreit} apply in general).  The curve can be defined over $\Q(\lambda_7)$ and its automorphism group can be defined over an at most quadratic extension of $\Q(\lambda_7)$ ramified only at $13$.

\subsubsection*{Genus 15, $(2,3,9)$, $\PGL_2(\F_8) (=\SL_2(\F_8))$}

Elkies \cite[\S 2]{Elkies237} computed an equation for the genus $1$ curve $X_0(2,3,9;2)$, which happens to be an elliptic curve: it is the curve \textsf{162b3}:
\[ y^2+xy+y=x^3-x^2-95x-697. \]





\subsubsection*{Genus 16, $(3,4,6)$, $\PGL_2(\F_5)$}

By Theorem A, the field $M(X)$ it is a degree $d_{X} \leq 2$ extension of $\Q(\lambda_3,\lambda_4,\lambda_6)^{\la\Frob_5\ra}=\Q$.  But since $G \cong \PGL_2(\F_5)$, we have $d_{(X,G)}=1$ so $M(X,G)=M(X)$.

This example is interesting because the quaternion algebra $A$ is defined over $E=\Q(\sqrt{6})$ and ramified at the prime $(\sqrt{6}+2)$ over $2$.  The field $E$ has narrow class number $2$ though class number $1$---the extension $\Q(\sqrt{-2},\sqrt{-3})$ over $\Q(\sqrt{6})$ is ramified only at $\infty$.  This implies that the Shimura curve is in fact a disjoint union of two curves with an action of this strict class group.  The group $\Delta(3,4,6)$ is obtained by an Atkin-Lehner quotient, and since the prime above $2$ represents the nontrivial class in the strict class group, the involution interchanges these two curves; consequently, the quotient is something that will be defined canonically over over $E=\Q(\sqrt{6})$ .

The genus $1$ curve $X_0=X_0(3,4,6;5)$ can be computed as follows.  The ramification data above the designated points $0,1,\infty$ is $3^2, 4^1 1^2, 6^1$.  We take the point above $\infty$ to be the origin of the group law on $X_0$ and take the point $(0,1)$ to be the ramification point of order $4$.  The curve $X_0$ is then described by an equation
\[ y^2 = x^3 + \lambda_2 x^2 + \lambda_1 x + 1 = f(x) \]
and the map $\phi(x,y)=a_0+a_1x+a_2x^2+a_3x^3 + (b_0+b_1x)y = a(x)+b(x)y$ is of degree $6$.  By ramification, we must have
\[ N\phi(x,y) = a(x)^2-b(x)^2 f(x) = a_3^2 x^4 (x^2+c_1x+c_0) \]
and
\[ N(\phi(x,y)-1) = (a(x)-1)^2 - b(x)^2 f(x) = a_3^2 (x^2+d_1x+d_0)^3 \]
for some values $c_0,c_1,d_0,d_1$.  We solve the corresponding system of equations for the values $a_0,\dots,\lambda_2$ and find a unique solution defined over $E=\Q(\sqrt{6})$: after simplifying, the elliptic curve $X$ has minimal model
\[ X_0:y^2 + (\sqrt{6} + 1)xy = x^3 + (-8\sqrt{6} + 23)x + (-2\sqrt{6} + 7) \]
with $j$-invariant
\[ j(X_0)=\frac{-21355471243\sqrt{6}+55502166112}{2^5 5^9} \in \Q(\sqrt{6}) \setminus \Q, \]
showing that $M(X)=M(X,G)=\Q(\sqrt{6})$, and
\begin{align*}
 \phi(x,y) &= (186\sqrt{6} + 351)x^3 + (9504\sqrt{6} + 21564)x^2 + (25350\sqrt{6} + 58725)x + 11280\sqrt{6} + 26730 \\
&\qquad +((1611\sqrt{6} + 4401)x + (7602\sqrt{6} + 18882))y
\end{align*}
with ramification at $0,50000,\infty$.












\subsubsection*{Genus 24, $(4,5,6)$, $\PGL_2(\F_5)$}

Finally we arrive at the first of two nonarithmetic curves.

Elkies has computed another interesting subcover of genus $1$, namely, the curve $E$ corresponding to the permutation representation of $\PGL_2(\F_5)$ as $S_5$: it is the curve
\[ y^2 + (17+2\sqrt{6})xy + 36(7-3\sqrt{6})y = x^3 - 36(1+\sqrt{6})x^2 \]
and the \Belyi\ function is
\[ \phi(x,y) = xy + (-9+6\sqrt{6})x^2 + (117-48\sqrt{6})y \]
with a pole of degree $5$ at infinity, a zero at $P=(0,0)$ of degree $4$, and
taking the value $2^8 3^3 (-5+2\sqrt{6})$ with multiplicity 3 at $-6P = (12(\sqrt{6}-1),-144)$ and multiplicity 2 at
$9P = (12(9-4\sqrt{6}), 48(27-10\sqrt{6}))$.

This was computed as follows.  The ramification type is $4\ 1, 5, 3\ 2$.  Put the point of ramification index $5$ as the origin of the group law and the $4$ point above $0$; let $P$ and $P'=-4P$ be the preimages with multiplicities $4,1$.  The preimages of multiplicity $3$ and $2$ are $2Q$ and $-3Q$ for some point $Q$.  But the divisor of $d\phi/\omega$, where $\omega$ is a holomorphic differential, is $3P+2(2Q)+(-3Q)-6\infty$ hence $Q=-3P$, so these preimages are $-6P$ and $9P$.

We take $E:y^2+a_1xy+a_3y=x^3+a_2x^2$ so that $P$ is at $(0,0)$ with a horizontal tangent; we scale so $a_2=a_3$ and let $a_1=a+1$ and $a_2=a_3=b$.  Then $\phi=axy-bx^2+by$.  The condition that $f-f(-6P)$ vanishes to order at least $2$ at $-6P$ and vanish at $9P$ leaves a factor $a^2-216a+48$ giving $(a,b)=(108+44\sqrt{6}, -(6084+2484\sqrt{6}))$.

\subsubsection*{Genus 24, $( 3, 4, 7 )$, $\PSL_2(\F_7)$}

We have $F=F(3,4,7)=\Q(\lambda_6,\lambda_8,\lambda_{14})=\Q(\sqrt{2},\lambda_7)$ and $E=E(3,4,7)=\Q(\lambda_3,\lambda_4,\lambda_7,\lambda_6\lambda_8\lambda_{14})=\Q(\sqrt{2},\lambda_7)=F$, and $F_7(3,4,7)=\Q$.

The curve $X(3,4,7;7)$, according to Theorem A, is defined over an at-most quadratic extension $M(X)$ of $\Q$; the curve with its automorphism group is defined over $M(X)(\sqrt{-7})$.

We find that the curve $X_0(3,4,7;7)$ has genus one and is defined over $\Q(\sqrt{2})$:
\[ E: y^2 + (\sqrt{6}+1)xy + (\sqrt{6}+1)y = x^3 + (-14\nu+3)x+(-78\nu+128) \]
with $j$-invariant
\[ j(E)=-\frac{14000471420\sqrt{6} +19227826689}{2^3 3^4 7^5} \]
and good reduction away from $2,3,7$.  The specialization $\phi=-1$ gives a Galois extension of $\Q(\sqrt{2},\sqrt{-7})$ with Galois group $\PSL_2(\F_7)$.  This gives reason to believe that $M(X)=\Q(\sqrt{2})$ and $M(X,G)=\Q(\sqrt{2},\sqrt{-7})$.

Another genus 1 quotient of $X(3,4,7;7)$ was computed by Elkies by ``turning a $p$-adic crank'' (explained in detail in the final example).  This curve is defined only over $\Q(\sqrt{2},\sqrt{-7})$, and already the $j$-invariant of the curve
\[ \frac{-28505956008\sqrt{2}\sqrt{-7} + 39863931701\sqrt{-7} + 120291604664\sqrt{2} +   15630829689}{107495424} \]
generates this field.  The subgroup $H \leq \PSL_2(\F_7)$ that gives rise to this curve is not stable under $\sigma:\sqrt{-7} \mapsto -\sqrt{-7}$; however, this elliptic curve is $2$-isogenous to its Galois conjugate by $\sigma$, and so defines a $\Q(\sqrt{2})$-curve, giving further evidence for this hypothesis.

\medskip

We conclude with an example which extends beyond our table but illustrates an important point.

\begin{exm}
Consider the triple $(3,5,6)$ and the prime $11$.  We obtain from Theorem B a $\PSL_2(\F_{11})$-cover, and this is the curve with smallest genus in the list of all $\PSL_2$-covers such that $a \neq 2$ and $p \nmid abc$.  This is the extreme case of Theorem A, where no hypothesis allows us to reduce $d_X$ or $d_{(X,G)}$ or deduce something about $M(X)$ from $M(X,G)$.

We compute that
\[ F_{11}(3,5,6)^{\la \Frob_{11} \ra} = \Q(\lambda_3,\lambda_5,\lambda_6)^{\la \Frob_{11}\ra}=\Q(\lambda_5) = \Q(\sqrt{5}) \]
and so $M(X)$ is a degree at most $2$ extension of $\Q(\lambda_5)$ quadratic over
\[ E(3,5,6)=\Q(\lambda_3,\lambda_5,\lambda_6,\lambda_6\lambda_{10}\lambda_{12})=\Q(\sqrt{3},\sqrt{5}). \]

The curve $X_0(3,5,6;11)$ has genus $2$, and so offers additional computational difficulties.  We instead compute the curve $E$ which arises from the quotient by the subgroup $A_5 \subseteq \PSL_2(\F_{11})$.  This subcover $E \to \PP^1$ is of genus $1$ and has degree $11=\#\PSL_2(\F_{11})/\#A_5$.

An equation for this curve was computed by Sijsling and the second author \cite{SijslingVoight}: we find a single Galois orbit of curves defined over the field $\Q(\sqrt{5},\sqrt{3},\sqrt{b})$ where
\[ b=4\sqrt{3} + \frac{1}{2}(11+\sqrt{5}); \]
with $N(b)=11^2$; alternately, it is given by extending by a root $\beta$ of the equation
\[ T^2 - \frac{1+\sqrt{5}}{2}T - (\sqrt{3}+1) = 0. \]
The elliptic curve has minimal model:
\begin{footnotesize}
\begin{align*}
&y^2 + ((1/2(13\sqrt{5} + 33)\sqrt{3} + 1/2(25\sqrt{5} + 65))\beta +
    (1/2(15\sqrt{5} + 37)\sqrt{3} + (12\sqrt{5} + 30)))xy \\
&\quad + (((8\sqrt{5} + 15)\sqrt{3} + 1/2(31\sqrt{5} +
    59))\beta + (1/2(13\sqrt{5} + 47)\sqrt{3} + 1/2(21\sqrt{5} + 77)))y  \\
&\quad = x^3 + ((1/2(5\sqrt{5} +
    7)\sqrt{3} + 1/2(11\sqrt{5} + 19))\beta + (1/2(3\sqrt{5} + 17)\sqrt{3} + (2\sqrt{5} + 15)))x^2  \\
&\quad +
    ((1/2(20828483\sqrt{5} + 46584927)\sqrt{3} + 1/2(36075985\sqrt{5} + 80687449))\beta \\
&\qquad    +(1/2(21480319\sqrt{5} + 48017585)\sqrt{3} + 1/2(37205009\sqrt{5} + 83168909)))x  \\
&\quad +
    (((43904530993\sqrt{5} + 98173054995)\sqrt{3} + 1/2(152089756713\sqrt{5} +
    340081438345))\beta \\
&\qquad +((45275857298\sqrt{5} + 101240533364)\sqrt{3} + (78420085205\sqrt{5} +
    175353747591)))
\end{align*}
\end{footnotesize}

The $j$-invariant of this curve generates the field $\Q(\sqrt{5},\sqrt{3},\sqrt{\beta})$.  So this gives strong evidence for our Theorem A, and shows that it is ``best possible'' in that one may indeed obtain a nontrivial extension in one of the unknown quadratic extensions which arise from the failure of (weak) rigidity.
\end{exm}

\begin{rmk}
Elkies has suggested that further examples could be obtained by considering triangle covers which are arithmetic, but not congruence.
\end{rmk}

\end{document}